\documentclass[12pt]{amsart}
\usepackage{amsmath}
\usepackage{amsfonts,amssymb}
\usepackage{amsthm}
\usepackage[matrix,arrow,curve]{xy}
\usepackage{array}

\numberwithin{equation}{section}
\theoremstyle{plain}
\newtheorem{theorem}{Theorem}[section]
\newtheorem{lemma}[theorem]{Lemma}
\newtheorem{prop}[theorem]{Proposition}
\newtheorem{corollary}[theorem]{Corollary}

\theoremstyle{definition}
\newtheorem{definition}[theorem]{Definition}
\newtheorem{remark}[theorem]{Remark}
\newtheorem{example}[theorem]{Example}

\newcommand{\s}{\sigma}

\newcommand{\N}{\mathbb N}
\newcommand{\Z}{\mathbb Z}

\renewcommand{\P}{\mathbb P}
\newcommand{\X}{\mathbb X}
\renewcommand{\AA}{\mathcal A}
\newcommand{\BB}{\mathcal B}
\newcommand{\CC}{\mathcal C}
\newcommand{\DD}{\mathcal D}
\newcommand{\TT}{\mathcal T}
\renewcommand{\O}{\mathcal O}
\renewcommand{\k}{\mathsf k}

\newcommand{\xra}{\xrightarrow}
\renewcommand{\le}{\leqslant}
\renewcommand{\ge}{\geqslant}

\newcommand{\bul}{\bullet}
\newcommand{\mmod}{\mathrm{{-}mod}}
\newcommand{\modd}{\mathrm{mod{-}}}
\newcommand{\Modd}{\mathrm{Mod{-}}}

\DeclareMathOperator{\Hom}{\textup{Hom}}
\DeclareMathOperator{\Ext}{\textup{Ext}}

\DeclareMathOperator{\End}{\mathrm{End}}
\DeclareMathOperator{\Spec}{\mathrm{Spec}}
\DeclareMathOperator{\coker}{\mathrm{coker}} 
\DeclareMathOperator{\im}{\mathrm{im}}
\DeclareMathOperator{\coh}{\mathrm{coh}}

\DeclareMathOperator{\Perf}{\mathrm{Perf}}
\DeclareMathOperator{\id}{\mathrm{id}}
\DeclareMathOperator{\Sdim}{Sdim}
\DeclareMathOperator{\LSdim}{\underline{\mathrm{Sdim}}}
\DeclareMathOperator{\USdim}{\overline{\mathrm{Sdim}}}
\DeclareMathOperator{\Rdim}{Rdim}
\DeclareMathOperator{\gldim}{\mathrm{gldim}}
\DeclareMathOperator{\Ddim}{Ddim}
\DeclareMathOperator{\op}{\mathrm{op}}

\DeclareMathOperator{\Tot}{\mathrm{Tot}}

\begin{document}

\title[Calculating dimension of triangulated categories]{Calculating dimension of triangulated categories: path algebras, their tensor powers and orbifold projective lines}

\author{Alexey ELAGIN}
\address{Institute for Information Transmission Problems (Kharkevich Institute), Moscow, RUSSIA;
National Research University Higher School of Economics, Russian Federation}
\email{alexelagin@rambler.ru}

\begin{abstract}
This is a companion paper of \cite{EL}, where different notions of dimension for triangulated categories are discussed. Here we compute dimensions for some examples of triangulated categories and thus illustrate and motivate material from \cite{EL}. Our examples include path algebras of finite ordered quivers, orbifold projective lines, some tensor powers of path algebras in Dynkin quivers of type $A$ and categories, generated by an exceptional pair.
\end{abstract}

\maketitle

\tableofcontents

\section{Introduction}

In \cite{EL} V.\,Lunts and the author define and discuss several notions of dimension for triangulated categories: Rouquier  dimension, diagonal dimension, lower and upper Serre dimension. They prove some general properties of these dimensions, make some conjectures and pose open questions.
In this paper we study some explicit examples of triangulated categories and calculate their dimension. Some of our examples illustrate general phenomena mentioned in \emph{loc. cit.}, motivate open questions and give evidence for conjectural answers. 
On the other hand, some other examples are of their own interest. 

Let us recall the definitions given in \emph{loc.\,cit.} (see Section~\ref{section_dimensions} for details). We restrict ourselves to triangulated categories of the form $\TT=D^b(\modd A)$ where $A$ is a finite-dimensional smooth algebra over some field $\k$. By abuse of language, for such categories we speak about the dimension of the \emph{algebra} instead of its derived category of modules. Rouquier dimension $\Rdim(\TT)$ is the minimal generation time  over all classical generators of~$\TT$. Diagonal dimension $\Ddim(A)$ is the minimal time needed to generate the diagonal $A^{\op}\otimes_\k A$-bimodule $A$ by some object of the form $F\boxtimes G$ where $F\in D^b(\modd A^{\op})$, $G\in D^b(\modd A)$. Let $S\colon D^b(\modd A)\to D^b(\modd A)$ be the Serre functor and $S^n$ be its $n$-th iteration. Then lower and upper Serre dimensions of $A$ are defined as
$$\LSdim(A)=\lim_{n\to+\infty}\frac{-\max\{i\mid H^i(S^n(A))\ne 0\}}n,\quad
\USdim(A)=\lim_{n\to+\infty}\frac{-\min\{i\mid H^i(S^n(A))\ne 0\}}n.$$

The following ``test problems'' are addressed in \emph{loc.\,cit.}:
\begin{enumerate}
\item classify categories of dimension zero;
\item is dimension monotonous in semi-orthogonal decompositions?
\item how does dimension behave under tensor product of categories?
\item how does dimension behave in families of categories?
\end{enumerate}
In addition to these problems it is natural to ask:
\begin{enumerate}
\item[(5)] what is the relation between different notions of dimension?
\end{enumerate}
For some pairs (a problem, a version of dimension) the solution is given in \emph{loc.\,cit.}, while for other pairs we only can guess the answer  but not prove it. 
Examples studied in this paper suggest answers to some of the above questions and motivate conjectures and expectations from  \emph{loc.\,cit.}

Let us describe these answers and conjectures. We work over a fixed field $\k$.

\medskip
(1) Classification of categories of zero Rouquier dimension or diagonal dimension is given in \emph{loc.\,cit.} In Section~\ref{section_twovertices} we demonstrate a family  of  smooth and compact dg algebras~$A$ (in fact, just graded algebras) such that the category $\Perf A$ has negative lower Serre dimension. Moreover, $\LSdim$ can be  arbitrarily small for such categories. 

For the algebra of dual numbers one has $\LSdim(\k\langle t\rangle)=\USdim(\k\langle t\rangle)=0$, but this algebra is not smooth. In Examples~\ref{example_2} and~\ref{example_3} we show that $\LSdim(A)$ can be zero for (ordinary) finite-dimensional algebra $A$ of finite global dimension. On the other hand, we expect that upper Serre dimension of $\Perf A$ is non-negative for any smooth and compact dg algebra~$A$, and can be zero only in trivial cases. In particular, we expect that $\USdim(A)>0$ for any finite-dimensional algebra $A$ with $0<\gldim(A)<\infty$. 

Let $A$ be  a path algebra of a non-trivial finite ordered quiver with relations. We prove in Proposition~\ref{prop_positive} that the lower Serre dimension of~$A$ is positive.
That is, vanishing of $\LSdim(\TT)$ is an obstruction for a triangulated category $\TT$ (like in Examples~\ref{example_2} and~\ref{example_3}) to be equivalent to $D^b(\modd A)$, where $A$ is a path algebra of a non-trivial finite ordered quiver with relations. 

(2) We say that dimension is monotonous in semi-orthogonal decompositions if $\dim \AA\le \dim \TT$ for any admissible subcategory $\AA\subset \TT$. Rouquier and diagonal dimensions are monotonous, see \emph{loc. cit.} Lower Serre dimension is not monotonous, as can be seen by taking direct products of triangulated categories. Indeed, for  $\TT=\TT_1\times\TT_2$ one has $\LSdim (\TT)=\min (\LSdim (\TT_1),\LSdim (\TT_2))$. Hence, if $\LSdim (\TT_1)<\LSdim (\TT_2)$ then $\LSdim (\TT_2)>\LSdim (\TT)$ and monotonuity fails. 
Neither upper Serre dimension is monotonous. In \cite[Ex. 5.15]{EL} it is shown that the category $\Perf A_V$ from Section~\ref{section_twovertices}  with $V=\k[-1]\oplus \k\oplus \k[1]$ is embeddable into the bounded derived category of coherent sheaves on Hirzebruch surface~$\mathbb F_3$. By Proposition~\ref{prop_ww} we have $\USdim(\Perf A_V)=3$ while $\USdim(D^b(\coh \mathbb F_3))=\dim \mathbb F_3=2$.

(3) Let $A,B$ be smooth and compact dg algebras over $\k$. Then for upper and lower Serre dimensions one has $\dim (\Perf A\otimes_\k B)=\dim(\Perf A)+\dim (\Perf B)$, see \emph{loc. cit.} Also, for diagonal dimension it is demonstrated in \emph{loc. cit.} that
$\Ddim (\Perf A\otimes_\k B)\le \Ddim(\Perf A)+\Ddim (\Perf B)$. Our Proposition~\ref{prop_maintable} shows that this can be a strict inequality: one can take $A=B$ to be the path algebra of the quiver $\bul\to\bul$.
For Rouquier dimension, Proposition~\ref{prop_maintable} provides several examples of algebras such that $\Rdim(A),\Rdim ( B)$ and $\Rdim ( A\otimes_\k B)$ are known. In all of them one has $\Rdim ( A\otimes_\k B)\ge \Rdim(A)+\Rdim ( B)$, in some cases the inequality is strict. We expect that inequality 
$\Rdim (\Perf A\otimes_\k B)\ge \Rdim(\Perf A)+\Rdim (\Perf B)$ holds for all smooth and compact dg algebras $A,B$.

(4) Let $R$ be a commutative Noetherian ring. By a smooth family of algebras over $\Spec R$ we mean a smooth algebra $A$ over $R$ which is a projective $R$-module of finite rank. For a point $x\in \Spec R$ the fiber $A_x$ of a smooth family $A$ of algebras is defined as $A\otimes_R \k(x)$, it is a smooth finite-dimensional $\k(x)$-algebra. We prove in \emph{loc. cit.} that the 
function $\LSdim (A_x)$ (resp. $\USdim(A_x)$) is lower (resp. upper) semi-continuous on $\Spec R$. 

Consider $R=\k[t]$, let $A$ be the path algebra of the quiver
$$
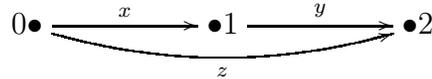
\xymatrix{0 \bul  \ar[rr]^x \ar@/_1pc/[rrrr]_z && \bul 1 \ar[rr]^y && \bul 2
}$$
over $R$ with one relation $yx-tz=0$. This defines a smooth family over $\k[t]$. Its general fiber $A_t, t\ne 0$ is isomorphic to the path algebra of the quiver $\bul\to\bul\to\bul$, the special fiber~$A_0$ is studied in Example~\ref{example_1}. 
We have  the following dimensions:
$$
\begin{array}{|l||c|c|c|c|}
		 \hline
			 & \Rdim & \Ddim & \LSdim  & \USdim\\
		 \hline
			A_t, t\ne 0 & 0 & 1 & 1/2& 1/2  \\
		 \hline	
		 A_0 & 1 & 1 & 1/2& 2 \\
		 \hline	
\end{array}
$$
That is, $\Rdim$ and $\USdim$ jump at special point.
This example supports our expectation that the function $\Rdim(A_t)$ should be upper semi-continuous on $\Spec R$.

(5) It follows from definitions that $\Rdim (\Perf A)\le \Ddim(\Perf A)$ and 
$\LSdim (\Perf A)\le \USdim(\Perf A)$ for any smooth compact dg algebra $A$. For the path algebras of Dynkin quivers (see Section~\ref{section_path}) dimensions are compared as follows
$$\Rdim<\LSdim=\USdim<\Ddim,$$
while for some categories from Section~\ref{section_twovertices} we have 
$$\LSdim<\Rdim=\Ddim<\USdim.$$
Nevertheless, in all examples we have studied we have $\LSdim\le \Ddim$ and $\Rdim\le \USdim$. We expect that for any smooth and compact dg algebra $A$ and $\TT=\Perf A$ such that $\LSdim (\TT)=\USdim (\TT)$ the following inequalities hold:
$$\Rdim(\TT)\le\LSdim(\TT)=\USdim(\TT)\le \Ddim(\TT).$$

\medskip
In this paper we concentrate on studying categories $D^b(\modd A)$, where $A$ is a path algebra of a finite quiver with relations. There is one example of geometrical origin that should be mentioned separately. In Section~\ref{section_orbifold} we study orbifold projective lines (in the sense of Geigle and Lenzing \cite{GL}). We prove 
\begin{theorem}[See Prop. \ref{prop_orbifold}]
\label{theorem_orbifold}
For any orbifold projective line $\X$ we have 
$$\LSdim(D^b(\coh\X))=\USdim(D^b(\coh\X))=\Rdim(D^b(\coh\X))=1.$$
\end{theorem}
This result agrees with a conjecture by Orlov saying that $\Rdim(D^b(\coh X))=\dim (X)$ for any smooth variety (or stack) $X$.

\medskip
Let us now describe the content and the structure of the paper.

In Section~\ref{section_prelim} we recall necessary background material (quivers, path algebras, their representations, derived and triangulated categories, exceptional collections and their mutations, derived Morita-equivalence) and introduce notation. 

In Section~\ref{section_dimensions} we give definitions of dimension from~\cite{EL} and formulate their properties that we will need. 

In Section~\ref{section_path} we study path algebras of finite ordered quivers. 
We compute all four dimensions for such algebras.
\begin{theorem}[See Propositions {\ref{prop_DynkinCY}, \ref{prop_dynkin}, \ref{prop_nondynkin}, \ref{prop_ddimnorels}}]
Let $\Gamma$ be a connected ordered quiver. Then dimensions of the category $D^b(\modd \k\Gamma)$ are given by
$$
\begin{array}{|l||c|c|c|c|}
		 \hline
			 & \Rdim & \Ddim & \LSdim  & \USdim\\
		 \hline
			\text{$\Gamma$ is a Dynkin quiver with the Coxeter number $h$} & 0 & 1 & \frac{h-2}h & \frac{h-2}h  \\
		 \hline	
		 \text{$\Gamma$ is not a Dynkin quiver} & 1 & 1 & 1& 1 \\
		 \hline	
\end{array}
$$
\end{theorem}

In Section~\ref{section_pathrels} we consider path algebras of finite ordered quivers with relations. Its main result is Proposition~\ref{prop_positive} saying that   the lower Serre dimension of such algebra is always positive. 

In Section~\ref{section_orbifold} we deal with orbifold projective lines. We prove Theorem~\ref{theorem_orbifold} here.

In Section~\ref{section_powers}  we study tensor powers of path algebras in Dynkin quivers of type $A$. Let~$B_m$ be the path algebra of the quiver $\bul\to\bul\to\ldots\to 
\bul$ ($m$ vertices). Let 
$$B_m^n=(B_m)^{\otimes n}=B_m\otimes B_m\otimes\ldots\otimes B_m.$$
Our interest in these algebras has two origins. First,  algebras $B_m^n$ provide examples of tensor products where Rouquier dimension can be computed. These examples can help to understand the relation between Rouquier dimension of a tensor product and of its factors.
Second, algebras $B_2^3, B_3^2$ and $B_2\otimes B_5$ are of special interest because of their relation with orbifold projective lines (from our point of view, they are distinguished by the property $\USdim=1$). These three algebras arise in non-commutative geometry, singularity theory and mirror symmetry. We will touch this relation in a forthcoming paper.

Serre dimension of algebras $B_m^n$ is known due to its multiplicativity:
$$\LSdim(B_m^n)=\USdim(B_m^n)=\frac{m-1}{m+1}\cdot n.$$
We have computed Rouquier dimension in the following cases, see Proposition~\ref{prop_maintable}:
$$
\begin{array}{|r||c|c|c|c|c||c|c|c|c|}
		 \hline
			 B_m^n & B_2  & B_2^2  &B_2^3  &B_2^4  &B_2^6  & B_3  &B_3^2  &B_3^3 & B_3^4  \\
		 \hline
			\Rdim(B_m^n) & 0 & 0 & 1& 1& 2 & 0 & 1& 1& 2\\
		 \hline	
\end{array}
$$

In Section~\ref{section_twovertices} we study ``graded'' quivers with two vertices. Consider a quiver  with two vertices and several arrows spanning a vector space $V$:
$$\bul\stackrel{V}{\Longrightarrow} \bul.$$
Put some $\Z$-grading on $V$ and denote the corresponding graded path algebra by $A_V$, we treat it as a dg algebra with zero differential.  
We prove 
\begin{theorem}[See Prop.~\ref{prop_ww}]
Assume $\dim(V)\ge 2$, then for the category $\Perf A_V$ we have
$$
\begin{array}{|c|c|c|c|}
		 \hline
			 \Rdim & \Ddim & \LSdim & \USdim \\
		 \hline
			1& 1& 1-w & 1+w\\
		 \hline	
\end{array}
$$
where $w=\max\{i\mid V^i\ne 0\}-\min\{i\mid V^i\ne 0\}$ is the ``width'' of $V$.
\end{theorem}
Note that any (dg enhanced) triangulated category, generated by an exceptional pair $(E_1,E_2)$, is equivalent to the category $\Perf A_V$ for $V=\Hom^\bul(E_1,E_2)$.

In Section~\ref{section_other} we consider three particular path algebras with relations.
Let $A$ be the path algebra of the quiver
$$
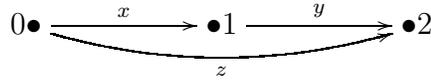
\xymatrix{0 \bul  \ar[rr]^x \ar@/_1pc/[rrrr]_z && \bul 1 \ar[rr]^y && \bul 2}
$$
with relation $yx=0$. Let $B$ be the path algebra of the quiver 
$$
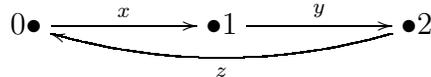
\xymatrix{0 \bul  \ar[rr]^x  && \bul 1 \ar[rr]^y && \bul 2 \ar@/^1pc/[llll]^z}
$$
with relations $zy=xz=0$. Let $C$ be the path algebra of the quiver 
$$
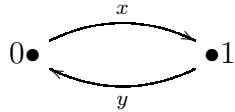
\xymatrix{0 \bul  \ar@/^1pc/[rr]^x  && \bul 1 \ar@/^1pc/[ll]^y }
$$
with relation $xy=0$. 
We have  then the following dimensions, see Examples~\ref{example_1},~\ref{example_2}, ~\ref{example_3}:
$$
\begin{array}{|r|c|c|c|c|}
\hline
 & \Rdim & \Ddim & \LSdim  & \USdim\\
\hline
 A & 1 & 1 & 1/2& 2 \\
\hline	
 B & 1 & \text{?} & 0& 3 \\
\hline	
 C & 1 & 1 & 0 	& 2 \\
\hline	
\end{array}
$$

\medskip
I thank Valery Lunts for his motivating interest which initiated this study and made its results written down. Most of results from this paper were obtained in course of and because of collaboration with him. I am also grateful to Dmitry Orlov for making me aware of graded algebras from Section~\ref{section_twovertices}.

\section{Preliminary material}
\label{section_prelim}

We work over a fixed field $\k$, which can be arbitrary.

\subsection{Quivers}
For definition and facts about algebras, quivers and their representations we refer to \cite{Ri} or \cite{ARS}.

By definition, a quiver $\Gamma$ is an oriented graph. It consists of a set $\Gamma_0$ of vertices and a set $\Gamma_1$ of edges (also called arrows). For any arrow $a\in \Gamma_1$ we denote by $s(a)$ its \emph{source} and by $t(a)$ its \emph{target}. In this paper we consider only finite quivers, that is, the sets $\Gamma_0$ and $\Gamma_1$ are finite. A quiver $\Gamma$ is called \emph{ordered} if a linear order on $\Gamma_0$ is chosen such that for any arrow $a$ one has $s(a)<t(a)$. Clearly, a quiver can be ordered if and only if it has no oriented cycles.

A \emph{path} from $v_0\in \Gamma_0$ to $v_n\in\Gamma_0$ is by definition a sequence of arrows
$a_1,\ldots,a_n$ such that $s(a_1)=v_0$, $t(a_n)=v_n$ and $t(a_k)=s(a_{k+1})$ for all $k=1,\ldots, n-1$. Such path is written as $p=a_na_{n-1}\cdot \ldots \cdot a_1$, it is said to have source $s(p)=v_0$, target $t(p)=v_n$ and length $n$.  Clearly, any arrow is a path of length $1$. Also, by definition for any $v\in\Gamma_0$ we have a path of length $0$ from $v$ to $v$, denoted by $e_v$.

The path algebra $\k\Gamma$ of $\Gamma$ id defined as follows. As a $\k$-vector space, it has the basis formed by all paths in $\Gamma$. The composition law is defined on basic elements $p=a_n\ldots a_1$ and $q=b_m\ldots b_1$ by $pq=a_n\ldots a_1b_m\ldots b_1$ if $t(q)=s(p)$ and $pq=0$ otherwise. The algebra $\k\Gamma$ is associative, it has the identity  $1=\sum_{v\in \Gamma_0}e_v$, it is finite-dimensional if and only if $\Gamma$ has no oriented cycles. The algebra $\k\Gamma$ is usually non-commutative. Also, $\k\Gamma$ has a grading by the path length: $\k\Gamma=\oplus_{i\ge 0}(\k\Gamma)_i$.

Let $R=R(\k\Gamma)=\oplus_{i>0}(\k\Gamma)_i\subset \k\Gamma$ be the subspace spanned by paths of positive length. Clearly, $R$ is a two-sided ideal. If $\Gamma$ has no oriented cycles then $R$ is the radical of $\k\Gamma$. By \emph{relations} in $\Gamma$ we mean a family of elements in $R^2$, or a two-sided ideal $I$ in $\k\Gamma$, generated by such family. Clearly, one can always choose a finite set of ``homogeneous'' relations: any homogeneous  relation has the form $\sum_{i=1}^n \lambda_i p_i$ where $\lambda_i\in\k$ and $s(p_1)=\ldots=s(p_n)$, $t(p_1)=\ldots=t(p_n)$. 
For an ideal of relations $I$ a \emph{path algebra with relations} is defined as $\k\Gamma/I$. We usually use the same notations for elements in $\k\Gamma$ and their images in $\k\Gamma/I$. Note that the restriction of the quotient map $\k\Gamma\to\k\Gamma/I$ to the vector space $(\k\Gamma)_0\oplus (\k\Gamma)_1=\langle e_v, a\rangle_{v\in\Gamma_0, a\in\Gamma_1}$ is an isomorphism.

Let $A=\k\Gamma/I$. For any $u,v\in\Gamma_0$ we denote $A_{uv}=e_uAe_v\subset A$, it is a linear subspace and we have $A=\oplus_{u,v\in \Gamma_0}A_{uv}$. Clearly, $A_{uv}$ is the space of paths from $v$ to $u$ modulo some relations. In particular, if $\Gamma$ is ordered then
$A_{uv}=0$ for $u<v$.

For a quiver $\Gamma$ we define its \emph{length} $l(\Gamma)$ as the maximal length of paths in $\Gamma$. Equivalently, it is the minimal number $l$ for which there exists a division of $\Gamma_0$ into $l+1$ groups $\Gamma_{0,0},\Gamma_{0,1},\ldots,\Gamma_{0,l}$ such that for any arrow $a\in \Gamma_1$ one has $s(a)\in\Gamma_{0,i}, t(a)\in\Gamma_{0,j}$ where $i<j$.
Clearly, $l(\Gamma)<\infty$ if and only if $\Gamma$ can be ordered.

We say that a quiver $\Gamma$ is \emph{a tree} if it is a tree as a non-oriented graph. In other words, it means that $\Gamma$ is connected and has no cycles and loops.


\subsection{Representations of algebras}

By an algebra we mean an associative unital algebra over $\k$.
By modules and representations in this paper we always mean \textbf{right} modules/representations unless stated otherwise explicitly.
For an algebra $A$, by 
$\modd A$ (resp. $A\mmod$) we denote the category of finitely generated right (resp. left) $A$-modules, by 
$\Modd A$  we denote the category of all right $A$-modules.  Assume $A$ is finite dimensional, then the category $\modd A$  has Krull-Schmidt property: any module is a finite direct sum of indecomposable modules, and such decomposition is unique in the following sense: if $\oplus_{i=1}^n M_i\cong \oplus_{i=1}^m N_i$ are two such decompositions then $n=m$ and up to renumbering of modules one has $M_i\cong N_i$ for all $i$.

A finite dimensional $\k$-algebra $A$ is called \emph{basic} if the quotient $A/R(A)$ of $A$ modulo its radical is isomorphic to an  algebra $\k\times\ldots \times \k$.

Algebras $A$ and $B$ over $\k$ are called (right) \emph{Morita-equivalent} if there exists a $\k$-linear equivalence of categories $\Modd A\to \Modd B$. 

Recall classical results due to Gabriel. (1) For any basic finite dimensional algebra $A$ there exists a finite quiver $\Gamma$ and an ideal $I$ of relations satisfying $(R(\k\Gamma))^n\subset I\subset (R(\k\Gamma))^2$ for some $n$, such that $A$ is isomorphic to the finite-dimensional algebra $\k\Gamma/I$. Moreover, such quiver $\Gamma$ is uniquely defined. (2) If $\k$ is algebraically closed, then any finite-dimensional $\k$-algebra $A$ is Morita-equivalent to a uniquely defined basic finite-dimensional $\k$-algebra. That is, representation theory of finite-dimensional algebras is the same as representation theory of quivers with relations.

\subsection{Representations of quivers}

Let $A=\k\Gamma/I$ be a path algebra of quiver $\Gamma$ with some relations. Further we assume that $A$ is finite dimensional, it is always the case if $\Gamma$ has no oriented cycles.
By a \emph{representation} of quiver $\Gamma$ (resp. quiver $\Gamma$ with relations~$I$) we mean a right module over $\k\Gamma$ (resp. over $A$). 

Let $M$ be a right module over $A$. For any $v\in \Gamma_0$ we denote by $M_v$ the subspace $Me_v\subset M$. One has a decomposition of $\k$-vector spaces
$$M=\oplus_{v\in \Gamma_0} M_v.$$
Any arrow $a$ from $u$ to $v$ induces a map $M_a\colon M_v\to M_u$.

Equivalently, any collection of vector spaces $M_v, v\in\Gamma_0$ and linear maps $M_a\colon M_{t(a)}\to M_{s(a)}, a\in \Gamma_1$ defines a right module over $\k\Gamma$. If the maps $M_a$ obey relations defining an ideal $I\subset \k\Gamma$ then the above data defines a right module over $\k\Gamma/I$.

For any $v\in \Gamma_0$ we consider the submodule $P_v=e_vA\subset A$. One has 
$$A\cong \oplus_{v\in \Gamma_0} P_v,$$
hence modules $P_v$ are projective. Moreover, the modules $P_v$ are pairwise non-isomorphic, indecomposable and any indecomposable projective module is isomorphic to one of $P_v$.
For any $A$-module $M$ one has 
$$\Hom_A(P_v,M)=M_v\qquad\text{and}\qquad\Hom_A(P_u,P_v)=A_{vu}.$$
Similarly, for any $v\in \Gamma_0$ we consider the projective left $A$-modules $Q_v=Ae_v\subset A$. 
The modules $Q_v, v\in \Gamma_0$ are all (up to an isomorphism) indecomposable projective left $A$-modules.
For any left $A$-module $M$ one has 
$$\Hom_A(Q_v,M)=e_vM\qquad\text{and}\qquad\Hom_A(Q_u,Q_v)=A_{uv}.$$

There is an equivalence of categories 
$$\modd A\to (A\mmod)^{\op},$$
given by $M\mapsto M^*=\Hom_{\k}(M,\k)$.
Under this equivalence the projective left modules $Q_v$ correspond to the injective right modules $I_v:=Q_v^*$. 
For any $A$-module $M$ one has 
$$\Hom_A(M,I_v)=\Hom_A(Q_v,M^*)=M_v^*\qquad\text{and}\qquad\Hom_A(I_u,I_v)=A_{vu}.$$

Also, for any $v\in \Gamma_0$ we consider the module $S_v=\k$: the idempotent $e_v$ acts on $S_v$ by identity and any other path acts by zero. The modules $S_v, v\in\Gamma_0$ are simple and pairwise non-isomorphic. 

If $\Gamma$ has no oriented cycles then the modules $P_v,I_v$ and $S_v$ are exceptional: one has $\Hom_A(M,M)=\k$ and $\Ext_A^i(M,M)=0$ for $i\ne 0$. Also, any simple module is isomorphic to one of $S_v$.

If $\Gamma$ has length $l$ then the algebra $\k\Gamma/I$ has global dimension $\le l$ for any relations $I$. The path algebra $\k\Gamma$ has global dimension $1$ (unless there are no arrows and the algebra $\k\Gamma\cong \k^n$ is semi-simple).

\subsection{Triangulated and derived categories, exceptional collections, mutations}

For the definitions related with triangulated and derived categories see \cite{BK1}, \cite{Bo}, \cite{Ne}.

Let $\TT$ be a $\k$-linear triangulated category. For any objects $X,Y\in\TT$ and $i\in\Z$ we use the standard notation $\Hom^i(X,Y):=\Hom(X,Y[i])$. By 
$\Hom^{\bul}(X,Y)$ we denote the graded vector space $\oplus_{i\in \Z}\Hom^i(X,Y)$ (where $\Hom^i(X,Y)$ has degree $i$).

A $\k$-linear triangulated category $\TT$ is said to be \emph{$\Hom$-finite} (resp. $\Ext$-finite) if for any $X,Y\in\TT$ the $\k$-vector space $\Hom(X,Y)$ (resp. $\Hom^{\bul}(X,Y)$) is finite-dimensional.

A full triangulated subcategory of a triangulated category is called \emph{thick} if it is closed  under taking direct summands.

Recall that an object $E\in\TT$ is called \emph{exceptional} if $\Hom(E,E)=\k$ and $\Hom^i(E,E)=0$ for any $i\ne 0$. An ordered  collection $(E_1,\ldots,E_n)$ of objects in $\TT$ is called \emph{exceptional} if any $E_k$ is exceptional and $\Hom^i(E_k,E_l)=0$ for all $i$ and $k>l$. An exceptional collection $(E_1,\ldots,E_n)$ is said to be \emph{strong} if $\Hom^i(E_k,E_l)=0$ for any $i\ne 0$ and any $k,l$. 
An exceptional collection $(E_1,\ldots,E_n)$ is called \emph{full} if the minimal full strict triangulated subcategory in $\TT$ containing $E_1,\ldots,E_n$ is $\TT$.

Let $(E_1,\ldots,E_n)$ be an exceptional collection. It is said that the subset 
 $(E_p,E_{p+1},\ldots,E_q)$ forms \emph{a block} if $\Hom^i(E_k,E_l)=0$ for any $i$ and $p\le k<l\le q$. We write down such collections as 
 $$(E_1,\ldots,E_{p-1},\begin{matrix} E_p\\ \vdots \\ E_q\end{matrix}, E_{q+1},
 ,\ldots, E_n).$$

Let $(E,F)$ be an exceptional pair in a triangulated category $\TT$. Suppose $\TT$ is $\Ext$-finite. Then one can define left and right mutation of $(E,F)$ as exceptional pairs $(L_E(F),E)$ and $(F,R_F(E))$, where
\begin{gather*} 
L_E(F):= Cone( \Hom^{\bul}(E,F)\otimes E\to F);\\
R_F(E):=Cone(E\to \Hom^{\bul}(E,F)^*\otimes F)[-1].
\end{gather*}
Let $(E_1,\ldots,E_n)$ be an exceptional collection in $\TT$. 
By definition, for $i=2,\ldots,n$ its $i$-th left mutation is the exceptional collection
$$(E_1,\ldots, E_{i-2}, L_{E_{i-1}}(E_i),E_{i-1}, E_{i+1}, \ldots, E_n)$$
and for $i=1,\ldots,n-1$ its $i$-th right mutation is 
the exceptional collection
$$(E_1,\ldots, E_{i-1}, E_{i+1}, R_{E_{i+1}}(E_i),E_{i+2}, \ldots,  E_n).$$
Note that the subcategory generated by an exceptional collection does not change under mutations:
\begin{multline*}
\langle E_1,\ldots, E_{i-2}, L_{E_{i-1}}(E_i),E_{i-1}, E_{i+1}, \ldots, E_n\rangle=\\
=\langle E_1,\ldots,E_n\rangle=\\
=\langle E_1,\ldots, E_{i-1}, E_{i+1}, R_{E_{i+1}}(E_i),E_{i+2}, \ldots,  E_n\rangle.
\end{multline*}

Let $(E_1,\ldots,E_n)$ be an exceptional collection.
We denote the iterated mutations  by 
$$L_{E_1,\ldots,E_{n-1}}(E_n):=L_{E_1}(L_{E_2}(\ldots (L_{E_{n-1}}(E_n)))), \quad
R_{E_2,\ldots,E_n}(E_1):=R_{E_n}(R_{E_{n-1}}(\ldots (R_{E_2}(E_1)))).
$$
Suppose the objects $E_1,\ldots,E_n$ in an exceptional collection $E_1,\ldots,E_n,F$ form a block, then the total left mutation
can be computed in one step, one has
\begin{equation}
\label{eq_mutleft}
L_{E_1,\ldots,E_n}(F)=Cone (\oplus_i \Hom^{\bul}(E_i,F)\otimes E_i\to F).
\end{equation}
Similarly, 
suppose that  $(E,F_1,\ldots,F_n)$ is an exceptional collection and the objects $F_1,\ldots,F_n$ form a block. Then 
\begin{equation}
\label{eq_mutright}
R_{F_1,\ldots,F_n}(E)=Cone (E\to \oplus_i \Hom^{\bul}(E,F_i)^*\otimes F_i)[-1].
\end{equation}

\medskip
Let $A$ be a $\k$-algebra. We denote by $D(A)=D(\Modd A)$ the undounded derived category of all $A$-modules and by $\Perf A\subset D(A)$ its full subcategory of perfect complexes, that is, complexes which are  quasi-isomorphic to bounded complexes of finitely generated projective $A$-modules. If $A$ is Noetherian (for example, finite-dimensional), we denote by $D^b(\modd A)$ the bounded derived category of finitely generated $A$-modules and by $D^b_{fg}(A)$ the full subcategory in $D(A)$ formed by complexes whose cohomology modules are finitely generated and almost all are zero. There is a natural equivalence $D^b(\modd A)\to D^b_{fg}(A)$. 
Also, there is an inclusion 
$$\Perf A\subset D^b_{fg}(A),$$
which is an equality if $A$ has finite global dimension. Hence, for Noetherian algebra $A$ of finite global dimension the categories $\Perf A$ and $D^b(\modd A)$ are equivalent. We prefer to use notation $\Perf A$ in this case.

The categories $D(A), \Perf A, D^b_{fg}(A), D^b(\modd A)$ are triangulated, Karoubian and $\k$-linear. 
Let $A$ be a finite-dimensional algebra, then  the category $\Perf A$ is $\Ext$-finite; the category $D^b(\modd A)$ is $\Ext$-finite if and only if $A$ has finite global dimension.

Suppose $A=\k\Gamma/I$ is a path algebra of some quiver with relations and $\Gamma$ has no oriented cycles. If we order $\Gamma_0=\{1,2,\ldots,n\}$ compatibly with arrows then the collection  of projective modules
$P_1,\ldots,P_n$ is full and strong exceptional in $D^b(\modd A)$, and one has 
$\End (\oplus_iP_i)\cong A$.

\subsection{Tilting}

Let $\TT$ be a triangulated category.
An object $G\in \TT$ is called  \emph{tilting}  if 
$\Hom^i(G,G)=0$ for all $i\ne 0$. Moreover, $G$ is called a   \emph{tilting generator} of $\TT$ if $G$ is tilting and $\TT$ is the smallest thick triangulated subcategory of $\TT$ containing $G$.

For example, if $(E_1,\ldots,E_n)$ is a strong exceptional collection in $\TT$ then the object $\oplus E_i\in \TT$ is tilting. We will frequently use the following standard result.
\begin{prop}
\label{prop_tilting}
Let $\TT$ be a triangulated Karoubian $\k$-linear category with a dg enhancement. For example, $\TT$ can be any  thick triangulated subcategory in the derived category $D(B)$ for a $\k$-algebra $B$. Assume $(E_1,\ldots,E_n)$ is a strong exceptional collection in $\TT$. Let $E=E_1\oplus\ldots \oplus E_n$ and $A:=\End_\TT E$.
Then 
there is an equivalence 
$$\langle E_1,\ldots,E_n\rangle \xra\sim \Perf A$$ 
given by the well-defined functor $R\Hom(E,-)$. Moreover, 
the algebra 
$A$ is basic and has finite global dimension.
\end{prop}

%

\begin{definition}
Algebras $A$ and $B$ over $\k$ are called (right) \emph{derived Morita-equivalent} if there exists a $\k$-linear equivalence of categories $\Perf A\to\Perf B$. By \cite[6.4]{Ke} or \cite{Rick}, $A$ and $B$ are derived Morita-equivalent iff the categories $D(A)$ and $D(B)$ are equivalent, and iff  the categories $D^b(\modd A)$ and $D^b(\modd B)$ are equivalent (the latter holds under the assumptions that $A$ and $B$ are Noetherian).
\end{definition}

We will also need a weaker notion.
\begin{definition}
\label{def_dsoc}
We say that an algebra $A$ is a \emph{derived semi-orthogonal subalgebra} of an algebra $B$ if  $\Perf A$ is $\k$-linearly equivalent to an admissible subcategory in $\Perf B$. 
\end{definition}

\begin{lemma}
\label{lemma_gena}
Let $A$ and $B$ be two $\k$-algebras. Then

\begin{enumerate}
\item $A$ and $B$ are derived Morita-equivalent if and only if there exists a tilting generator $G\in \Perf B$ with $\End_B(G)\cong A$.

\item  Assume also that $A,B$ are finite-dimensional and $\gldim A<\infty$. Then 
$A$ is a derived semi-orthogonal subalgebra of  $B$ if and only if there exists a tilting object~$G$ in $\Perf B$ such that $\End_B(G)\cong A$.
\end{enumerate}
\end{lemma}
\begin{proof}
For (1) see \cite[6.1]{Ke}, let us explain (2).

For ``only if'' implication, assume $\phi\colon \Perf A\to \Perf B$ is a fully faithful functor and take $G:=\phi(A)$.
For ``if'' part, we have an equivalence $\langle G\rangle\to \Perf A$ given by the functor $R\Hom_B(G,-)$. It remains to check that the subcategory $\langle G\rangle$ is admissible in $\Perf B$. This follows from \cite[Prop. 3.17]{Or} since $\Perf B$ is $\Ext$-finite and $\Perf A$ has a strong generator.
\end{proof}

\begin{lemma}
\label{lemma_tensor}
Let $A,A',B,B'$  be $\k$-algebras. 
\begin{enumerate}
\item Assume $A$ is derived Morita-equivalent to $B$ and $A'$ is derived Morita-equivalent to $B'$.  Then $A\otimes_{\k} A'$ is derived Morita-equivalent to $B\otimes_{\k} B'$.
\item Suppose also that $A,A',B,B'$ are finite-dimensional algebras, and $A,A'$ are basic over $\k$ and have finite global dimension. Assume $A$ is  a derived semi-orthogonal subalgebra of $B$ and $A'$ is a derived semi-orthogonal subalgebra of $B'$.  Then $A\otimes_{\k} A'$ is a derived semi-orthogonal subalgebra of  $B\otimes_{\k} B'$.
\end{enumerate}
\end{lemma}
\begin{proof}
In both cases, by Lemma~\ref{lemma_gena} there exist tilting objects $G\in \Perf B$, $G'\in \Perf B'$ such that $\End(G)\cong A$, 
$\End(G')\cong A'$. Then $G\otimes G'$ is a tilting object in $\Perf B\otimes B$ with  
$\End(G\otimes H)\cong A\otimes A'$. 

In (1) it remains to note that $G\otimes G'$ is a generator of $\Perf B\otimes B'$. Indeed, $B\in \langle G\rangle$ and $B'\in \langle G'\rangle$, hence $B\otimes B'\in \langle G\otimes G'\rangle$. But $B\otimes B'$ generates $\Perf B\otimes B'$.

In (2) it remains to note that the algebra $A\otimes_\k A'$ has finite global dimension since $A$ and $A'$ are basic and have finite global dimension, see \cite[Th. 16]{Au}. Now we use Lemma~\ref{lemma_gena}.
\end{proof}

\begin{lemma}
\label{lemma_bondalec}
Let $B$ be a finite-dimensional algebra and $E_1,\ldots,E_n$ be a strong exceptional collection in $\Perf B$. Denote $A:=\End_B(\oplus E_i)$. Then $A$ is a derived semi-orthogonal subalgebra of $B$.
\end{lemma}
\begin{proof}
Follows from Proposition~\ref{prop_tilting} and Lemma~\ref{lemma_gena}(2), take $G:=\oplus E_i$.
\end{proof}

The main tool for deriving derived Morita-equivalences are mutations of exceptional collections.
\begin{definition}
\label{def_mutequiv}
Let us say that two $\k$-algebras $A$ and $B$ are \emph{mutation-equivalent} if there exists a $\k$-linear triangulated category $\TT$ with a dg enhancement and strong exceptional collections
$(E_1,\ldots,E_n)$ and $(E'_1,\ldots,E'_n)$ in $\TT$ such that $A\cong \End(\oplus E_i)$, $B\cong \End(\oplus E'_i)$, and $(E'_1,\ldots,E'_n)$ can be obtained from $(E_1,\ldots,E_n)$ by several mutations and shifts of objects. 
\end{definition}
\begin{remark}
\label{remark_mutequiv}
It follows from Proposition~\ref{prop_tilting} that
\begin{enumerate}
\item In Definition~\ref{def_mutequiv} one can always take $\TT=\Perf A$ and $E_i=P_i$ to be the indecomposable projective $A$-modules. 
\item Mutation equivalent algebras are automatically basic and have finite global dimension.
\item Mutation equivalent algebras $A$ and $B$ are derived Morita-equivalent since
$$\Perf A\cong \langle E_1,\ldots,E_n\rangle= \langle E'_1,\ldots,E'_n\rangle\cong \Perf B.$$
\end{enumerate}
\end{remark}

\section{Dimension of triangulated categories: Rouquier, diagonal, Serre}
\label{section_dimensions}

Here we recall notions of dimension for triangulated categories discussed in \cite{EL}. We focus on derived categories of finite-dimensional algebras. Note that all our definitions of dimension for algebras are derived Morita invariant.

\subsection{Rouquier dimension}

Let us recall some notions related to generation of triangulated categories. We refer to \cite{EL} for details.

Let $\TT$ be a triangulated category and $G\in \TT$ an object. Define full  subcategories $[G]_i$ and $\langle G\rangle_i\subset \TT$ as follows. Let $[G]_0$ be the full  subcategory formed by all finite direct sums of shifts of $G$. Let $[G]_k$ be the full
subcategory formed by such objects $F$ that there exists a distinguished triangle
$F_0\to F\to F_{k-1}\to F_0[1]$ with $F_0\in [G]_0$ and $F_{k-1}\in [G]_{k-1}$. Let $\langle G\rangle_k$ be the idempotent closure of $[G]_k$ in $\TT$. Denote
$$[G]:=\cup_k [G]_k\qquad\text{and}\qquad \langle G\rangle:=\cup_k \langle G\rangle_k,$$
they are full triangulated subcategories in $\TT$. An object $G$ is called a \emph{classical generator} of $\TT$ if $\TT=\langle G\rangle$. An object $G$ is called a \emph{strong generator} of $\TT$ if $\TT=\langle G\rangle_n$ for some $n$. Next definition was proposed by R.\,Rouquier in \cite{Rou}
\begin{definition}
Dimension of a triangulated category $\TT$ is the minimal number $n$ such that there exists an object $G\in\TT$ with $\TT=\langle G\rangle_n$. 
\end{definition}
We call the above dimension \emph{Rouquier dimension} and denote it by $\Rdim(\TT)$. 
\begin{definition}
If $A$ is a ring we write $\Rdim(A)$ for $\Rdim(\Perf A)$.
\end{definition}

Clearly, Rouquier dimension is derived Morita-invariant: if algebras $A$ and $B$ are derived Morita-equivalent then $\Rdim(A)=\Rdim(B)$. 

Here are some other properties of Rouquier dimension.


\begin{prop}[See {\cite[Prop. 7.16]{Rou}}]
\label{prop_rdimgeom}
Let $X$ be a reduced separated scheme of finite type over a field $\k$. Then $\Rdim (D^b(\coh X))\ge \dim X$. 
\end{prop}



\begin{prop}[See {\cite[Section 3.2]{EL}}]
\label{prop_rdimsubcat}
Let $\TT'\subset \TT$ be an admissible triangulated subcategory. Then $\Rdim(\TT')\le \Rdim(\TT)$.
In particular, if $A,B$ are algebras and $A$ is a derived semi-orthogonal subalgebra in $B$ then $\Rdim(A)\le \Rdim(B)$.
\end{prop}

\subsection{Diagonal dimension}

In \cite{EL} the diagonal dimension of an enhanced triangulated category $\TT$ is defined as follows. Suppose $\TT\cong \Perf A$ where $A$ is a dg algebra over~$\k$.  Consider the bifunctor
$$\boxtimes\colon \Perf A^{\op} \times \Perf A \to \Perf A^{\op}\otimes_{\k} A.$$
\begin{definition}
\label{def_ddim}
Let $A$ be a dg $\k$-algebra (or just a $\k$-algebra). Diagonal dimension of $\Perf A $ is the minimal integer $n$ such that for some objects $F\in \Perf A^{\op}$, $G\in \Perf A$ the diagonal $A^{\op}\otimes_{\k} A$-bimodule $A$ belongs to $\langle F\boxtimes G\rangle_n$. We denote diagonal dimension of $\Perf A$ by $\Ddim(A)$ .
\end{definition}

\begin{remark}
It follows from Proposition 4.3 in  \cite{EL} and Theorem 3.2 in \cite{Ke2} that diagonal dimension is derived Morita-invariant for ordinary $\k$-algebras $A$ and $B$: 
if there is an equivalence $\Perf A\cong \Perf B$ then $\Ddim(A)=\Ddim(B)$.

Also note that $\Ddim(A)$ is finite if and only if $A$ is a smooth $\k$-algebra, i.e., if $A\in\Perf(A^{\op}\otimes_{\k} A)$.
\end{remark}

Below we list some properties of diagonal dimension.

\begin{prop}[See {\cite[Prop. 4.10]{EL}}]
\label{prop_rdimddim}
Let $A$ be a $\k$-algebra. Then $\Rdim(A)\le \Ddim(A)$.
\end{prop}

\begin{prop}[See {\cite[Prop. 4.8]{EL}}]
\label{prop_ddimmult}
Let $A$ and $B$ be $\k$-algebras. Then
$$\Ddim  (A\otimes_\k B)\le \Ddim (A))+\Ddim (B).$$
\end{prop}

\begin{prop}[See {\cite[Prop. 4.12]{EL}}]
\label{prop_ddimsubcat}
Let $A$ and $B$ be $\k$-algebras such that $A$ is a derived semi-orthogonal subalgebra of $B$ (Definition~\ref{def_dsoc}). 
Then $\Ddim(A)\le \Ddim(B)$.
\end{prop}

\begin{prop}
\label{prop_ddimexcoll}
Let $A$ be a finite-dimensional basic $\k$-algebra of finite global dimension. Assume that $\Perf A$ has a full exceptional collection consisting of $n+1$ blocks:
$$\Perf A=\left\langle \begin{matrix}E_{0,1}, & E_{1,1}, & \ldots, & E_{n,1}\\
 & & \vdots & \\
E_{0,d_0}, & E_{1,d_1}, & \ldots, & E_{n,d_n}
\end{matrix}
\right\rangle.$$
Then 
$$\Ddim(A)\le n.$$
\end{prop}
\begin{proof}
Follows from \cite[Prop. 4.13]{EL}, note that $A$ is smooth.
\end{proof}

\begin{prop}
\label{prop_ddimradical}
Let $A$ be a finite-dimensional $\k$-algebra with the radical $R$. Suppose the algebra $A/R$ is separable over $\k$ (for example, this holds if $\k$ is perfect or if $A$ is basic). Suppose that $R^{d+1}=0$. Then for the $A^{\op}\otimes_k A$-module $A$ one has 
$$A\in \langle (A/R)\otimes_\k (A/R)\rangle_{d}\subset D(A^{\op}\otimes_\k A).$$
In particular, if moreover $A$ has finite global dimension then $\Ddim(A)\le d$.
\end{prop}
\begin{proof}
Our assumptions imply that the algebra $(A/R)^{\op}\otimes_\k (A/R)$ is semi-simple, see for example \cite[Prop. 3.9]{FD}.

Consider the filtration of the bimodule $A$ 
$$0=R^{d+1}\subset R^d\subset R^{d-1}\subset\ldots\subset R^2\subset R\subset A.$$
For any quotient $M_i:=R^i/R^{i+1}$, $i=0,\ldots, d$, one has $R\cdot M_i=M_i\cdot R=0$. Hence $M_i$ is an $(A/R)^{\op}\otimes_\k (A/R)$-module. Since $(A/R)^{\op}\otimes_\k (A/R)$ is a semi-simple algebra, we have $M_i\in\langle (A/R)\otimes_\k (A/R)\rangle_0$. Now it follows that $A\in \langle (A/R)\otimes_\k (A/R)\rangle_{d}$. 

If $A$ has finite global dimension then $A/R\in\Perf A, \Perf A^{\op}$ and thus $\Ddim(A)\le d$ by Definition~\ref{def_ddim}.
\end{proof}

\medskip
\begin{lemma}[See {\cite[Prop. 4.15]{EL}}]
\label{lemma_ddim1}
Let $A$ be a finite-dimensional basic $\k$-algebra, then $\Ddim(A)\le \gldim(A)$.
\end{lemma}
\begin{proof}
Denote $n:=\gldim(A)$.
Since $A$ is basic, we have $\End_A(M)=\k$ for any simple $A$-module. Hence assumptions of Lemma 7.2 in \cite{Rou} are satisfied  and the projective dimension of $A^{\op}\otimes_\k A$-module $A$ is $n$. Since any finitely generated projective $A^{\op}\otimes_\k A$-module is in $\langle A\otimes_\k A\rangle_0$, 
it follows that $A\in\langle A\otimes_\k A\rangle_n$. Thus $\Ddim(A)\le n$ by Definition~\ref{def_ddim}.
\end{proof}

\subsection{Serre dimension}

Recall that a functor $S\colon \TT\to\TT$ on a $\k$-linear $\Hom$-finite  triangulated category is called a \emph{Serre functor} if there exists a functorial isomorphism
$$\Hom(X,Y)\cong  \Hom(Y,S(X))^*$$
for all $X,Y\in\TT$. Such functor is unique up to an isomorphism if exists. 

Let $A$ be a finite-dimensional basic $\k$-algebra of finite global dimension. That is, $A$ is isomorphic to some path algebra with relations of a finite quiver. Then the Serre functor on $\Perf A$ (which is equivalent to $D^b(\modd A)$ in this case) is given by 
$$S(X)=X\otimes_A^L A^*,$$ 
where $A^*=\Hom_{\k}(A,\k)$ is treated as an $A^{\op}\otimes A$-bimodule. We call $A^*$ the \emph{Serre bimodule}. 
Decomposition $A=\oplus_i Ae_i=\oplus_i Q_i$ of left $A$-modules yields decomposition
$A^*=\oplus e_iA^*=\oplus I_i$ of right $a$-modules. In particular, 
\begin{equation}
\label{eq_SPI}
S(P_i)=P_i\otimes_A^LA^*\cong e_iA\otimes_AA^*\cong e_iA^*\cong I_i. 
\end{equation}

In \cite{EL} the authors define Serre dimension for 
arbitrary triangulated category $\TT$ with a classical generator $G$ and a Serre functor $S$. 
First we make the following notation.
\begin{definition}
\label{def_w}
For a graded vector space $V$ we denote 
$$\inf V:=\inf\{i\mid V^i\ne 0\},\qquad \sup V:=\sup\{i\mid V^i\ne 0\}.$$ 
Also we denote $w(V):=\sup V-\inf V$.

Similarly, for a complex $C$ we put $\inf C:=\inf H^{\bul}(C), \sup C:=\sup H^{\bul}(C)$.
\end{definition}

By the definition, the upper Serre dimension and the lower Serre dimension of $\TT$ are 
$$
\LSdim(\TT):=\liminf_{m\to+\infty} \frac{-\sup \Hom^{\bul}(G,S^m(G))}m,\qquad 
\USdim(\TT):=\limsup_{m\to+\infty} \frac{-\inf \Hom^{\bul}(G,S^m(G))}m.
$$
We refer to \cite{EL} for the motivation and related discussions.

Suppose $A$ is a finite-dimensional algebra, $\gldim(A)$ is finite, $\TT=\Perf A$ and $G=A$. Then  (see \cite[Prop. 5.5]{EL}) the above definition boils down to
\begin{equation}
\label{eq_infsup}
\LSdim(\Perf A)=\lim_{m\to+\infty} \frac{-\sup ((A^*)^{\otimes^L_Am})}m, \qquad 
\USdim(\Perf A)=\lim_{m\to+\infty} \frac{-\inf ((A^*)^{\otimes^L_Am})}m,
\end{equation}
where $(A^*)^{\otimes^L_Am}=A^*\otimes^L_A A^*\otimes^L_A\ldots \otimes^L_A A^*$ is the $m$-th derived tensor power of the Serre bimodule.
For such algebras we will simply write $\USdim(A)$ and $\LSdim(A)$ instead of
$\USdim(\Perf A)$ and $\LSdim(\Perf A)$.

One has the following bounds, see \cite[Rem. 5.4, Rem. 5.10, Prop.  5.13]{EL}.
\begin{prop}
\label{prop_serregldim}	
Let $A$ be a finite-dimensional algebra of finite global dimension. Then
$$0\le \LSdim(A)\le \USdim(A)\le \gldim(A).$$
\end{prop}

We point out that derived Morita-equivalent algebras have equal upper and lower Serre dimensions. Indeed, dimension is defined in categorical terms.


\medskip
Recall the following definition
\begin{definition}
Triangulated $\k$-linear category $\TT$ with a Serre functor $S$ is called \emph{$\frac mn$-fractionally Calabi-Yau (or fractionally CY)} if the iterated Serre functor $S^n$ is isomorphic to the shift functor $[m]$. We say that a $\k$-algebra $A$ is $\frac mn$-fractionally Calabi-Yau if such is the category $\Perf A$.
\end{definition}

\begin{lemma}
\label{lemma_fCYtensor}
Let $A,B$ be two basic finite-dimensional $\k$-algebras of finite global dimension. Suppose $A$ is $\frac{m_a}{n_a}$-fractionally CY and $B$ is $\frac{m_b}{n_b}$-fractionally CY. Then
$A\otimes_{\k} B$ is a $\frac{m_an_b+m_bn_a}{n_an_b}$-fractionally CY algebra.
\end{lemma}
\begin{proof}
Let $C:=A\otimes_\k B$.
Since $A,B$ are basic, $C$ is also basic and has finite global dimension by 
\cite[Th. 16]{Au}. Now the statement follows from isomorphisms
$$(A^*)^{\otimes_A^L n_a}\cong A[m_a], \quad (B^*)^{\otimes_B^L n_b}\cong B[m_b],\quad 
(C^*)^{\otimes_C^L i}\cong (A^*)^{\otimes_A^L i}\otimes_\k (B^*)^{\otimes_B^Li}$$
for any $i\ge 1$.
\end{proof}

For a fractionally CY algebra, the Serre dimension can be easily computed, we have
\begin{prop}
\label{prop_fCYsdim}
Let $A$ be a finite-dimensional algebra of finite global dimension. Suppose $A$ is a $\frac mn$-fractionally CY algebra. Then
$$\LSdim(A)=\USdim(A)=\frac mn.$$
\end{prop}
\begin{proof}
For any $j>0$ one has $(A^*)^{\otimes^L_Ajn}\cong A[jm]$. Therefore
$$\inf (A^*)^{\otimes^L_Ajn}=\sup (A^*)^{\otimes^L_Ajn}=-jm.$$ 
The statement clearly follows from \eqref{eq_infsup}.
\end{proof}

The next lemma follows easily from the classification of categories with Rouquier dimension zero, see \cite[Th. 3.7]{EL}. We prefer to give a simple independent proof.
\begin{lemma}
\label{lemma_rdimpositive}
Let $\TT$ be an $\Ext$-finite triangulated $\k$-linear category with a Serre functor. 
Assume $\TT$ is connected (=indecomposable into direct  product).
Suppose $\Rdim (\TT)=0$, then $\LSdim (\TT)=\USdim (\TT)$.
\end{lemma}
\begin{proof}
By  Krull-Schmidt theory, any object in  $\TT$ is a direct sum of indecomposable objects, and there are  only finitely many indecomposable objects in $\TT$ up to a shift. It follows that for any indecomposable object $M$ there exist some $n(M)>0$ and $m(M)$ such that $S^{n(M)}(M)\cong M[m(M)]$. Taking a certain multiple we can assume that  $n(M)=:n$ does not depend on $M$. 

Now let $M_1,M_2\in\TT$ be indecomposable and $m(M_1)\ne m(M_2)$. We claim that $\Hom^\bul(M_1,M_2)=0$. Indeed, 
\begin{multline*}
\Hom^i(M_1,M_2)\cong \Hom^i(S^n(M_1),S^n(M_2))\cong \Hom^i(M_1[m(M_1)], M_2[m(M_2)])=\\
=\Hom^{i+m(M_2)-m(M_1)}(M_1,M_2)
\end{multline*}
and iterating we get 
$$\Hom^i(M_1,M_2)\cong   \Hom^{i+t(m(M_2)-m(M_1))}(M_1,M_2)$$
for any $t\ge 0$. Since $\TT$ is $\Ext$-finite, this is possible only if $\Hom^i(M_1,M_2)=0$ for all $i$.
We get a direct product decomposition
$$\TT\cong \prod_{m\in\Z} \langle M\mid \text{$M$ indecomposable, $m(M)=m$}\rangle.$$
Since $\TT$ is connected we get that all indecomposable objects $M\in\TT$ have the same value $m(M)=:m$. It follows from definitions now that $\LSdim \TT=\USdim \TT=\frac mn$.
\end{proof}

\section{Path algebras}
\label{section_path}

Let $\Gamma$ be a connected quiver with no oriented cycles. Let $A=\k\Gamma$ be its path algebra. Then $A$ is finite-dimensional and $\gldim(A)=1$ (unless $\Gamma$ has no arrows and then $\gldim(A)$=0). Recall that we have $\Perf A\cong D^b(\modd A)$.

Properties of the category $\Perf A$ are quite different in two cases: Dynkin quivers (such that the underlying graph is of Dynkin types $A_n, D_n$ or $E_6,E_7,E_8$) and non-Dynkin quivers.


\begin{prop}[See {\cite[Prop. 3.1]{HI}} or {\cite[Theorem 4.1]{MY}}]
\label{prop_DynkinCY}
Let $\Gamma$ be a Dynkin quiver. Then  the category $\Perf \k\Gamma$ is  $\frac{h-2}h$-fractionally Calabi-Yau: $S^{h}\cong [h-2]$, 
where $h$ is the Coxeter number of $\Gamma$, see table
$$
\begin{array}{|r|c|c|c|c|c|}
		 \hline
			type & A_n & D_n & E_6 & E_7 & E_8\\
		 \hline
			h & n+1 & 2(n-1) & 12& 18 & 30 \\
		 \hline	
\end{array}
$$
Hence 
$$\LSdim(\k\Gamma)=\USdim(\k\Gamma)=\frac{h-2}{h}.$$
\end{prop}

The following Proposition is contained essentially in \cite[Lemma 2.15]{DHKK}.
\begin{prop}
\label{prop_nonDynkinSdim}
Let $\Gamma$ be a connected non-Dynkin quiver and $A=\k\Gamma$. Then one has
$$\LSdim(A)=\USdim(A)=1.$$
\end{prop}
\begin{proof}
Let $M$ be an indecomposable module, then $S(M)\in \modd A$ if $M$ is projective and 
$S(M)\in (\modd A)[1]$ otherwise. Therefore $S^k(M)=M_k[n_k]$ for all $k\ge 1$ and some modules $M_k$.
Suppose also that $M$ is projective, then $S(M)$ is injective. It is known (see, for example, \cite[2.1.14]{Ri}) that the preprojective and preinjective components of the Auslander-Reiten quiver of $\modd A$ do not intersect. It follows that none of the modules $M_k, k\ge 1$ are projective. Hence, $S(M_k)\cong M_{k+1}[1]$ and $n_k=k-1$ for $k\ge 1$. It follows that the complex $(A^*)^{\otimes^L_A k}$ of $A^{\op}\otimes A$-bimodules  is a bimodule shifted by $k-1$. Consequently, $\USdim(A)=\LSdim(A)=1$ by \eqref{eq_infsup}.
\end{proof}

We collect  the properties of $\Perf \k\Gamma$ (mostly well-known) in the following  two propositions.

\begin{prop}
\label{prop_dynkin}
Let $\Gamma$ be a connected quiver with no oriented cycles and $A=\k\Gamma$.
The following conditions are equivalent:
\begin{enumerate}
\item $\Gamma$ is a Dynkin quiver;
\item there exists finitely many indecomposable finitely generated $A$-modules; 
\item there exists finitely many indecomposable objects in $\Perf A$ up to a shift;
\item $\Rdim(A)=0$; 
\item $A$ is a fractionally Calabi-Yau algebra;
\item $\LSdim(A)=\USdim(A)<1$.
\end{enumerate}
\end{prop}
\begin{proof}
We sketch the proofs or recall the reference for the convenience of the reader.

Equivalence (1) $\Longleftrightarrow$ (2) is a classical result by P.\,Gabriel.

Equivalence (2) $\Longleftrightarrow$ (3) holds because any complex of $A$-modules is isomorphic to the direct sum of its cohomology.

Equivalence (3) $\Longleftrightarrow$ (4) holds because $D^b(\modd A)$ has  Krull-Schmidt property: decomposition of objects into indecomposable summands is unique.

Implications (1) $\Longrightarrow$ (5),(6) are by Proposition~\ref{prop_DynkinCY}.

Implications (5),(6) $\Longrightarrow$ (1) follow from Proposition~\ref{prop_nonDynkinSdim}.
Briefly, suppose $\Gamma$ is non-Dynkin. Then for any indecomposable projective module $P$ we have $S^k(P)=M_k[k-1]$ for all $k\ge 1$ and some modules $M_k$.  Hence, $\USdim(A)=1$ and $A$ is not fractionally Calabi-Yau. 
\end{proof}

The same arguments also prove

\begin{prop}
\label{prop_nondynkin}
Let $\Gamma$ be a connected quiver with no oriented cycles and $A=\k\Gamma$.
The following conditions are equivalent:

\begin{enumerate}
\item $\Gamma$ is not a Dynkin quiver;
\item there exist infinitely many indecomposable finitely generated $A$-modules; 
\item $\Rdim(A)=1$; 
\item $A$ is not a fractionally Calabi-Yau algebra;
\item $\LSdim(A)=\USdim(A)=1$.
\end{enumerate}
\end{prop}

For diagonal dimension we have
\begin{prop}
\label{prop_ddimnorels}
Let $\Gamma$ be a quiver without oriented cycles. Then $\Ddim(\k\Gamma)=1$ if $\Gamma$ has at least one arrow, and $0$ otherwise.
\end{prop}
\begin{proof}
It follows from Lemma~\ref{lemma_ddim1} that $\Ddim(\k\Gamma)\le \gldim(\k\Gamma)\le 1$.
To finish the proof, we use classification of categories with diagonal dimension zero, see \cite[Prop. 4.6]{EL}. 
\end{proof}

\medskip
We recall a notion of reflections, see \cite{BGP}. 
Let $v\in\Gamma_0$ be a \emph{source} (resp. $u\in\Gamma_0$ be a \emph{sink}):  it means that there are no arrows pointing at $v$ (resp. starting at $u$). Let $s_v^+\Gamma$ (resp. $s_u^-\Gamma$ ) be the quiver obtained from $\Gamma$ by inverting all arrows starting at $v$ (resp. ending at $u$). 

\begin{theorem}[See \cite{BGP}]
\label{theorem_BGP}
In the above notations the algebras $\k\Gamma$, $\k(s_v^+\Gamma)$  and $\k(s_u^-\Gamma)$ are mutation-equivalent.
\end{theorem}

\begin{proof}
We sketch the proof for $\k(s_v^+\Gamma)$. Let $P_i$, $i\in\Gamma_0$ be the indecomposable projective modules. 
Denote 
$$P'_v:=Cone\left(P_v\to \bigoplus_{a\in\Gamma_1: s(a)=v}P_{t(a)}\right).$$
Let $P'_i:=P_i$ for $i\ne v$.
Consider the collection of objects $(P'_i)_{i\in\Gamma_0}$ in $D^b(\modd \k\Gamma)$ with $P'_v$ put the last.
One checks that this collection is full and  strong exceptional and 
$\End(\oplus_i P'_i)\cong \k(s_v^+\Gamma)$. Also one checks that $P'_v=R_{P_i, i\ne v}(P_v)[1]$ (see \eqref{eq_mutright}), therefore the algebra $\k(s_v^+\Gamma)$ is mutation-equivalent to $\k\Gamma$ by definition.
\end{proof}

Operations $s^+$ and $s^-$ are called \emph{reflections}. The following corollary will be needed for studying examples in Section~\ref{section_powers}.

\begin{corollary}
\label{cor_forest}
Let $\Gamma$ be a tree and $\Gamma'$ be a quiver obtained from $\Gamma$ be inverting some arrows. Then $\k\Gamma$ is mutation-equivalent to $\k\Gamma'$. In particular,  $\k\Gamma$ is derived Morita-equivalent to $\k\Gamma'$.
\end{corollary}
\begin{proof}
It follows from Theorem~\ref{theorem_BGP} since  $\Gamma'$ can be obtained from $\Gamma$ by several reflections.
\end{proof}



\section{Path algebras of ordered quivers with relations}
\label{section_pathrels}

Let $\Gamma$ be an ordered quiver and $A=\k\Gamma/I$ be a path algebra \textbf{with relations}. In this section we present some general facts about dimension of such algebras.

In contrast with path algebras without relations, the lower and upper Serre dimension may be not equal, as examples 
from Section~\ref{section_other} demonstrate. 
What is good, Serre dimensions of such algebras are positive.

\begin{prop}
\label{prop_positive}
Let $A$ be a path algebra of a  nontrivial connected ordered quiver $\Gamma$ modulo some relations. Then $\LSdim(A)>0$. To be more precise,  let $l(\Gamma)$ be the length of~$\Gamma$, then
$\LSdim(A)\ge \frac 1{r}$, where $r=l(\Gamma)+2$.
\end{prop}
\begin{proof}
Clearly, $H^0((A^*)^{\otimes_A^Lr})=(A^*)^{\otimes_Ar}=0$, where the last equality is by Lemma~\ref{lemma_MMM} below. 
Then $\sup((A^*)^{\otimes_A^Lr})\le -1$ and hence $\sup((A^*)^{\otimes_A^Lrk})\le -k$. Now the claim follows from \eqref{eq_infsup}.
\end{proof}

\begin{remark}
There exist finite-dimensional algebras of finite global dimension which are not  Morita-equivalent (and moreover are not derived Morita-equivalent) to a path algebra with relations in an ordered quiver. Indeed, the algebra $A$ from Example~\ref{example_2} has  $\LSdim(A)=0$ what contradicts to Proposition~\ref{prop_positive}.
\end{remark}

\begin{lemma}
\label{lemma_MMM}
Let $A$ be a path algebra of a  nontrivial connected ordered quiver $\Gamma$ modulo some relations. Let $l(\Gamma)$ be the length of $\Gamma$. Then one has
$$(A^*)^{\otimes_Ar}=A^*\otimes_A\ldots \otimes_AA^*=0,$$
where $r=l(\Gamma)+2$.
\end{lemma}
\begin{proof}
Recall that by assumptions $\Gamma_0$ is partially ordered such that for any arrow $a\in\Gamma_1$ one has $s(a)<t(a)$. Also recall that we treat $A^*=\Hom_{\k}(A,\k)$ as an $A^{\op}\otimes A$-module with the multiplication
$$(a\cdot f\cdot b)(x):=f(bxa), \quad\text{where}\quad
a,b,x\in A, f\in A^*.$$
Denote $(A^*)_{ij}:=e_i\cdot A^*\cdot e_j$ for any $i,j\in\Gamma_0$.
One can easily see that $(A^*)_{ij}=(A_{ji})^*$, thus $(A^*)_{ij}=0$ unless $i\le j$.
Clearly, for $f\in (A^*)_{ij}$ and $a\in A_{kl}$ one has
$$fa\in (A^*)_{il},\qquad fa=0 \quad\text{if}\quad j\ne k, $$ 
and similarly
$$af\in (A^*)_{kj},\qquad af=0 \quad\text{if}\quad l\ne i.$$ 
Denote by $e^i\in (A^*)_{ii}$ the element dual to $e_i\in A_{ii}$.

Let $i<j$ be vertices and $f\in (A^*)_{ij}$,  $a\in A_{ji}$ be such that $f(a)=1$. Then 
\begin{equation}
\label{eq_affa}
f\cdot a=e^i,\qquad a\cdot f=e^j.
\end{equation}
Indeed,
$e_k(fa)e_l=(e_k f)\cdot(ae_l)=0$ unless $k=i$ and $i=l$, therefore $fa\in (A^*)_{ii}$. 
Now $(fa)(e_i)=f(ae_i)=f(a)=1$, hence $fa=e^i$. Similarly $af=e^j$.

We claim that for any $f\in (A^*)_{ij}$ one has
\begin{equation}
\label{eq_effe}
e^i\otimes f=f\otimes e^j\in A^*\otimes_AA^*.
\end{equation}
Indeed, choose $a\in A_{ji}$ such that $f(a)=1$. Then we have
$$
e^i\otimes f=fa\otimes f=f\otimes af=f\otimes e^j,
$$
where the first and the last equalities are (\ref{eq_affa}), while the middle one is by the definition of $\otimes_A$.

Also we claim that for any $i\in\Gamma_0$
\begin{equation}
\label{eq_eiei}
e^i\otimes e^i=0\in A^*\otimes_AA^*.
\end{equation}
Indeed, since $\Gamma$ is connected and nontrivial, there exists an arrow $a\in\Gamma_1$ such that either $s(a)=i$ or $t(a)=i$.
In the first case let $j:=t(a)$, choose $f\in (A^*)_{ij}=(A_{ji})^*$ such that $f(a)=1$.
We have 
$$e^i\otimes e^i=fa\otimes e^i=f\otimes ae^i=0,$$
where the first equality is by (\ref{eq_affa}) and the third one is because $ae^i\in (A^*)_{ji}=0$ (as $j>i$).
The case $t(a)=i$ is treated similarly.

Now we check that $(A^*)^{\otimes_A r}=0$. Indeed, consider any nonzero element
$$f=f_1\otimes f_2\otimes \ldots \otimes f_r\in (A^*)^{\otimes_A r}.$$
We can suppose that $f_k$ are homogeneous: $f_k\in (A^*)_{i_kj_k}$.

First, we have $j_k=i_{k+1}$ for all $k=1,\ldots,r-1$.  Indeed, otherwise
$$f_k\otimes f_{k+1}=f_ke_{j_k}\otimes e_{i_{k+1}}f_{k+1}=f_k\otimes e_{j_k}e_{i_{k+1}}f_{k+1}=0.$$
Further, we can assume that $i_1\le i_2\le \ldots \le i_r\le j_r$ (otherwise some $f_k=0$).
Since the maximal length of a path in $\Gamma$ is $l(\Gamma)$ and $r=l(\Gamma)+2$, it follows that for some $p<q$, we have $i_p=j_p$ and $i_q=j_q$.  Then (up to a constant) $f_p=e^{i_{p+1}}$ and $f_q=e^{i_q}$.

Now we have 
\begin{multline*}
f_p\otimes f_{p+1}\otimes\ldots\otimes f_{q-1}\otimes f_q=
e^{i_{p+1}}\otimes f_{p+1}\otimes\ldots\otimes f_{q-1}\otimes e^{i_q}=\\
=f_{p+1}\otimes e^{i_{p+2}}\otimes f_{p+2} \otimes\ldots\otimes f_{q-1}\otimes e^{i_q}=
\ldots = 
f_{p+1}\otimes f_{p+2} \otimes\ldots\otimes f_{q-1}\otimes e^{i_q}\otimes e^{i_q}=0.
\end{multline*}
where we use (\ref{eq_effe}) and \eqref{eq_eiei}.
It follows now that $f=0$.
\end{proof}

\begin{remark}
We point out that the ``nilpotence degree'' $r=l(\Gamma)+2$  of the Serre bimodule~$A^*$ provided by Lemma~\ref{lemma_MMM}   is in some cases the minimal possible. For example, let $A$~be the path algebra of the linearly oriented quiver 
$$\Gamma_n:\qquad 1\xra{d} 2\xra{d}\ldots \xra{d} n$$ 
with relations $d^2=0$. Then one has $S(P_i)\cong I_i\cong P_{i+1}$ for any $i=1,\ldots, n-1$.
It follows that $S^n(P_1)\cong I_n$ and $(A^*)^{\otimes_A n}\ne 0$. Here we have $l(\Gamma_n)=n-1$.

At the same time, the bound $\LSdim(A)\ge \frac 1r$ is almost never exact. In the above example we have $\LSdim(A)=\USdim(A)=\frac{n-1}{n+1}$ since $A$ is derived Morita-equivalent to $\k\Gamma_n$, see Proposition~\ref{prop_DynkinCY}.
\end{remark}

\medskip

For Rouquier dimension and diagonal dimension we have  obvious 
\begin{prop}
Let $\Gamma$ be an ordered quiver of length $n$.
Let $A$ be a path algebra of~$\Gamma$ modulo some relations. Then $\Rdim(A)\le \Ddim(A)\le \gldim(A)\le n$.
\end{prop}
\begin{proof}
It follows from Proposition~\ref{prop_rdimddim} and Lemma~\ref{lemma_ddim1}. 
\end{proof}

\medskip
We will need the following lemma for studying our examples in Sections \ref{section_powers} and~\ref{section_other}.
\begin{lemma}
\label{lemma_timesdyn}
Let $\Gamma$  be an ordered quiver of length $n$,  and $A=\k\Gamma/I$ for some ideal of relations. Let  $\Delta$ be a  Dynkin quiver. Then $\Rdim(A\otimes \k\Delta)\le n$.
\end{lemma}
\begin{proof}
The algebra $A\otimes \k\Delta$ is a path algebra of quiver $\Gamma\times \Delta$ with some relations. 
Suppose $\Gamma_0=\sqcup_{i=0}^n \Gamma_{0,i}$, where any arrow $a\in\Gamma_1$ goes from $\Gamma_{0,i}$ to $\Gamma_{0,j}$ with $i<j$.
Then $(\Gamma\times \Delta)_0=\sqcup_{i=0}^n (\Gamma_{0,i}\times \Delta_0)$. Let $G_i$ be the full subquiver in $\Gamma\times \Delta$ with vertices $\Gamma_{0,i}\times \Delta_0$. Any arrow in $(\Gamma\times \Delta)_1$ goes from $G_i$ to $G_j$ with $i\le j$. Moreover, any $G_i$ is a disjoint union of several copies of $\Delta$.
Let $\BB_i$ denote the subcategory $\langle P_v\rangle_{v\in G_i}\subset \Perf (A\otimes \k\Delta)$ generated by projective modules. 
Note that there are no relations on paths in $\Gamma\times\Delta$ lying in ``slices'' of the form $v\times\Delta$, $v\in \Gamma_0$. 
Therefore 
$$\BB_i\cong \Perf \k G_i\cong \Perf \k(\Delta\sqcup\ldots\sqcup \k\Delta)\cong (\Perf \k\Delta )\times\ldots \times (\Perf \k\Delta)$$ and thus $\BB_i$ has Rouquier dimension $0$, see Proposition~\ref{prop_dynkin}. Also we have a semi-orthogonal decomposition
$$\Perf (A\otimes \k\Delta)=\langle \BB_0,\BB_1,\ldots,\BB_n\rangle.$$
It follows that  $\Rdim(A\otimes \k\Delta)\le n$.
\end{proof}


\section{Orbifold projective lines and canonical algebras}
\label{section_orbifold}

Here we study  orbifold projective lines or, equivalently, canonical algebras.

Let $V$ be a two-dimensional $\k$-vector space, $\P^1_{\k}=\P(V)$ be the projective line and $Q_1,\ldots,Q_n\in \P(V)$ be different points. Let $v_i\in V$ be corresponding vectors. For any collection of multiplicities $r_1,\ldots,r_n\ge 2$ a weighted projective line $\X=\X_{r_1Q_1,\ldots,r_nQ_n}$ is defined. Recall the definition following Geigle and Lenzing, see \cite{GL}. 

Denote by $L$ the abelian group generated by elements $\bar c,\bar x_1,\ldots, \bar x_n$ with the relations $\bar c-r_i\bar x_i=0$ for $i=1,\ldots, n$.   Note that $L$ is a partially  ordered abelian group, its set of positive elements is $\sum_i \N\bar x_i$. Choose nonzero elements $y_i\in V^*$ such that $(y_i,v_i)=0$ for $i=1,\ldots, n$. Let $S$ be the quotient algebra
$$S:=(S[V^*]\otimes_{\k}\k[X_1,\ldots,X_n])/(X_1^{r_1}-y_1,\ldots,X_n^{r_n}-y_n).$$
Consider $S$ as an $L$-graded algebra with the grading $\deg(V^*)=\bar c, \deg(X_i)=\bar x_i$. Let $\mathrm{mod}^L(S)$ and $\mathrm{mod}_0^L(S)$ denote the abelian categories of finitely generated $L$-graded $S$-modules and $L$-graded $S$-modules of finite length respectively. Then one defines the category of coherent sheaves on $\X$ as
the Serre quotient
$$\coh \X:=\mathrm{mod}^L(S)/\mathrm{mod}^L_0(S),$$
it is a $\k$-linear abelian hereditary category. The embedding of graded algebras $S[V^*]\to S$ defines a morphism $\X\to\P(V)$ of orbifolds and a pair of adjoint functors between $\coh \X$ and $\coh \P(V)$.

It is shown in \cite[Prop. 4.1]{GL} that the collection 
\begin{equation}
\label{eq_S}
S(\bar l)_{0\le \bar l\le \bar c}
\end{equation}
of twisted free $S$-modules $S(\bar l)$ (or line bundles $\O_{\X}(\bar l)$) is a full and strong exceptional collection in $D^b(\coh \X)$.
 
The corresponding endomorphism algebra is the path algebra of the quiver
$$\xymatrix{ && S(\bar x_1)\ar[r]^{X_1}& S(2\bar x_1)\ar[r]^-{X_1} & \ldots \ar[r]^-{X_1}& S((r_1-1)\bar x_1)\ar[rrd]^{X_1} && \\
  S \ar[rru]^{X_1}\ar[rr]^{X_2}\ar[rrdd]^{X_n} && S(\bar x_2)\ar[r]^{X_2}& S(2\bar x_2)\ar[r]^-{X_2} & \ldots \ar[r]^-{X_2}& S((r_2-1)\bar x_2)\ar[rr]^{X_2} && S(\bar c)\\
  &&&\ldots&\ldots\\
 && S(\bar x_n)\ar[r]^{X_n}& S(2\bar x_n)\ar[r]^-{X_n} & \ldots \ar[r]^-{X_n} & S((r_n-1)\bar x_n)\ar[rruu]^{X_n}&&
}$$
with relations coming from equalities $X_i^{r_i}=y_i\in\Hom(S,S(\bar c))=V^*$.
Denote this algebra by $A_{\X}$ or $A_{r_1Q_1,\ldots,r_nQ_n}$.
Such algebras are known as canonical algebras, see Ringel \cite[Section 3.7]{Ri}. 
One has an equivalence $D^b(\modd A_{\X})\cong D^b(\coh\X)$.

By \cite[(2.2)]{GL}, the Serre functor  on $D^b(\coh \X)$ is given by $S_{\X}(-)=-\otimes_{\O_{\X}} \omega_{\X}[1]$, where $\omega_{\X}=\O_{\X}((n-2)\bar c-\sum_i\bar x_i)$ is called a dualizing line bundle.
It follows that 
$$\LSdim(D^b(\coh \X))=\USdim(D^b(\coh \X))=1=\LSdim(A_{\X})=\USdim(A_{\X}).$$

In this section we demonstrate that $\Rdim(A_{\X})=\Rdim(D^b(\coh\X))=1$. This agrees with the expectation $\Rdim(D^b(\coh X))=\dim X$ for any smooth variety or stack $X$.

It is convenient for us to consider another full strong exceptional collection in $D^b(\coh \X)$, which can be obtained from (\ref{eq_S}) by a twist and some mutations. 
Explicitly, we take the collection  $(S(\bar l))_{-\bar c\le \bar l\le 0}$ and mutate all modules  right through $S$ except for $S(-\bar c)$ and $S$. The resulting collection
is
\begin{equation}
\label{eq_E}
\left(S(-\bar c),S,
\begin{smallmatrix}
S/(X_1^{r_1-1}),\ldots,S/(X_1^2),S/(X_1)\\ \ldots \\ 
S/(X_n^{r_n-1}),\ldots,S/(X_n^2),S/(X_n)
\end{smallmatrix}
\right).
\end{equation}
Modules $E_{i,j}:=S/(X_i^{j})$ correspond to torsion coherent sheaves on $\X$ supported over the points $Q_i\in\P(V)$. Modules $S(-\bar c)$ and $S$ correspond to the pull-backs of the sheaves $\O_{\P(V)}(-1)$ and $\O_{\P(V)}$ under the map $\X\to\P(V)$. 
The endomorphism algebra of (\ref{eq_E}) is the path algebra $\bar A_\X$ of the quiver 
$$\xymatrix{&& &&(1,1)\ar[r]& (1,2)\ar[r] & \ldots \ar[r]& (1,r_1-1)\\
0 \ar@{=>}[rr]^{V^*} && 1 \ar[rru]^{a_1}\ar[rr]^{a_2}\ar[rrdd]^{a_n} && (2,1)\ar[r]& (2,2)\ar[r] & \ldots \ar[r]& (2,r_2-1)\\
&&&&&\ldots&\ldots&\\
&& && (n,1)\ar[r]& (n,2)\ar[r] & \ldots \ar[r] & (n,r_n-1)
}$$
with relations $a_1y_1=\ldots =a_ny_n=0$.

We have $D^b(\modd A_{\X})\cong D^b(\coh \X)\cong D^b(\modd \bar A_{\X})$.

\begin{prop}
\label{prop_orbifold}
In the above notation one has  $\Rdim(D^b(\coh \X))=1$.
\end{prop}
\begin{proof}
The embedding $S[V^*]\to S$ of graded algebras corresponds to the morphism of orbifolds $p\colon \X\to \P(V)$, it produces  an adjoint pair $(p^*,p_*)$ of exact pull-back and push-forward functors between  $\coh \P(V)$ and $\coh \X$. Moreover, $p_*p^*\cong \id$, it follows that $p^*$ defines a fully faithful embedding 
$$p^*\colon D^b(\coh \P(V))\to D^b(\coh\X).$$
We get a semi-orthogonal decomposition
$$D^b(\coh \X)=(D^b(\coh \P(V)),^\perp (D^b(\coh \P(V))))=:(\CC,\DD),$$
which is compatible with exceptional collection~\eqref{eq_E}: one has 
$\CC=\langle S(-\bar c),S\rangle$ and 
$\DD=\langle E_{i,j}\rangle_{i,j}$.
It follows from Proposition~\ref{prop_rdimsubcat} that $\Rdim(\coh\X)\ge \Rdim(\CC)=\Rdim(D^b(\coh \P^1))=1$. Let us prove the opposite inequality.


We need the following lemma, which we prove later.

\begin{lemma}
\label{lemma_SCQ}
In the above notation, for any $C\in \CC$ there exists a triangle 
\begin{equation}
\label{eq_SCQ}
s(C)\xra{\s} C\to  q(C)\to s(C)[1]
\end{equation}
 such that
\begin{enumerate}
\item $s(C),q(C)\in\CC$;
\item for any morphism $f\colon C\to D$ where $D\in \DD$ one has $f\s=0$;
\item all irreducible summands of $s(C)$ and $q(C)$ belong (up to a shift) to a finite set of objects (independent on $C$).
\end{enumerate} 
\end{lemma}

Note that any object in $D^b(\coh \X)$ is a cone of a morphism $f\colon C\to D$
where $C\in\CC, D\in\DD$. Let $S\xra{\sigma} C\xra{\pi} Q\xra{\delta} S[1]$ be the triangle from Lemma~\ref{lemma_SCQ}. Since $f\s=0$, the morphism~$f$ factors through $Q$:
$f=f'\pi$. Note now that 
$$Cone (f)\cong Cone (Q\xra{(\delta,f')} S[1]\oplus D).$$
It remains to say that all irreducible summands of  $Q,S[1]$ and $D$ belong up to a shift to a finite list of objects. For $Q$ and  $S[1]$ it follows from Lemma~\ref{lemma_SCQ}.
For $D$, recall  that 
$$D\in \langle E_{i,j}\rangle_{i,j}\cong \langle P_{i,j}\rangle_{i,j}\subset D^b(\modd \bar A_\X)$$
(where $P_{i,j}$ denote projective $\bar A_\X$-modules).
The category $\langle P_{i,j}\rangle_{i,j}$ is equivalent to the direct product 
$$\prod_{i=1}^n\langle P_{i,1},\ldots,P_{i,r_i-1}\rangle.$$
Here each category is equivalent to the derived category of representations of the quiver of type $A_{r_i-1}$, hence contains only a finite number (up to a shift) of irreducible objects.

Thus there exists a finite set of objects that generate category $D^b(\coh \X)$ at one step and $\Rdim(D^b(\coh \X))=1$.
\end{proof}

\begin{proof}[Proof of Lemma~\ref{lemma_SCQ}]
First, note that it suffices to construct triangle~\eqref{eq_SCQ} for any indecomposable object $C\in \CC$. Second, any object in $\CC$ is of the form $p^*F$ with $F\in D^b(\coh \P(V))$. For any $D\in\DD$ we have
$$\Hom_{D^b(\coh\X)}(p^*F,D)\cong\Hom_{D^b(\coh \P(V))}(F,p_*D).$$
Recall that $\DD$ is generated by torsion coherent sheaves $E_{i,j}$ on $\X$ located at the orbifold points. It follows that $p_*D$ is supported on a finite set $Q_1,\ldots,Q_n\subset \P(V)$. Any such 
$p_*D$ belongs to the category $\langle \O_{Q_i}\rangle_{1\le i\le n}$.
We have reduced the statement of the Lemma to the following 

\textbf{Claim.} 
Let $Q_1,\ldots,Q_n\in\P(V)=\P^1_\k$ be some distinct closed points defined over $\k$. Then for any indecomposable coherent sheaf $F$ on $\P^1$ there exists a triangle in $D^b(\coh \P^1)$
\begin{equation*}
s(F)\xra{\sigma} F\to q(F)\to s(F)[1]
\end{equation*}
such that
\begin{enumerate}
\item for any morphism $f\colon F\to \O_{Q_i}[j]$ in $D^b(\coh \P^1)$  one has $f\s=0$;
\item all irreducible summands of $s(C)$ and $q(C)$ belong (up to a shift) to a finite set of coherent sheaves.
\end{enumerate} 
 
We prove this claim by considering all indecomposable coherent sheaves on $\P^1$. 

First, let $F=\O(k)$ be a line bundle.  Suppose $k\le -1$, then 
$F$ is quasi-isomorphic to the complex $[F^0\to F^1]=[\O(-1)^{-k}\to \O^{-k-1}]$. 
We put 
$$s(F)=F^1[-1]=\O^{-k-1}[-1], \quad q(F)=F^0=\O(-1)^{-k}.$$ 
Property (1) holds because for any nonzero $f\colon \O(k)\to \O_{Q_i}[j]$ we have $j=0$ and  $\Hom(s(F),\O_{Q_i})=\Ext^1(\O^{-k-1},\O_{Q_j})=0$.

Suppose  $0\le k\le n-1$, we put $s(F)=0$, $q(F)=F$ and there is nothing to check.

It is convenient to treat the remaining cases at once. Therefore we assume that $F$ is one of the following
\begin{enumerate}
\item $\O(k)$, $k\ge n$;
\item indecomposable torsion sheaf supported at some point $Q_a$, $1\le a\le n$;
\item indecomposable torsion sheaf supported at some other point $Q$. 
\end{enumerate}
Consider the following exact sequence 
\begin{equation}
\label{eq_U}
0\to U\xra\alpha F\xra\beta \bigoplus_{i=1}^n \Hom(F,\O_{Q_i})^*\otimes \O_{Q_i},
\end{equation}
where $\beta$ is the canonical map, $U=\ker \beta$ and $\alpha$ is the inclusion.
By writing this sequence explicitly one checks that $\beta$ is surjective, $U$ is generated by global sections and $H^1(\P^1,U)=0$.
Let $s(F):=H^0(\P^1,U)\otimes \O$ and $\s$ be the composition
$$\s\colon H^0(\P^1,U)\otimes \O\xra{e} U\xra\alpha F.$$
Let $q(F)$ be the cone of $\s$.

Let us check that property (1) holds. Any homomorphism $f\colon F\to \O_{Q_i}$ factors via $\beta$  hence $f\s=f'\beta\alpha e=0$. For $j\ne 0$ we have $\Hom^j(s(F),\O_{Q_i})=0$. 

For property (2) lets find irreducible summands of $s(F)$ and $q(F)$. For $s(F)$ we have only the sheaf $\O$. Note that $q(F)\cong (\ker \s)[1]\oplus \coker \s$. We have 
$\ker \s\cong \ker e$, by the above remarks one has $\Hom^\bul(\O,\ker e)=0$. Hence
$\ker e\in \langle \O(-1)\rangle$, and $\ker e$ (as a coherent sheaf) is isomorphic to $\O(-1)^d$ for some $d$. Finally we have $\coker \s\cong \coker \alpha\cong \im\beta$, this is a direct sum of sheaves $\O_{Q_i}$.

We have considered all cases. It remains to note that indecomposable summands of all $s(F)$ and $q(F)$ constructed above belong up to a shift to the following list of sheaves:
$$\O(-1),\O,\O(1),\ldots,\O(n-1),\O_{Q_1},\ldots,\O_{Q_n}.$$
\end{proof}

\section{Tensor powers of path algebras in Dynkin quivers of type $A$}
\label{section_powers}

Let $\Gamma_m$ be the Dynkin quiver of type $A_m$:
$$0\to 1\to 2\to \ldots \to m-1.$$
 In this section we study algebras 
$$B_m^n=(\k \Gamma_m)^{\otimes n}=(\k \Gamma_m)\otimes (\k \Gamma_m)\otimes\ldots\otimes (\k \Gamma_m)$$ (commutative $n$-dimensional cube with edge $m-1$), we denote $B_m:=B_m^1$.

\begin{lemma}
\label{lemma_sdimB}
We have for any $m\ge 2$
$$\gldim(B_m^n)=n; \qquad \LSdim(B_m^n)=\USdim(B_m^n)=n\cdot \frac {m-1}{m+1}.$$

Moreover, $B^n_m$ is a fractionally $\frac{n(m-1)}{m+1}$-Calabi-Yau algebra. 
\end{lemma}
\begin{proof}
The statement about global dimension follows from general theory and the fact that $\gldim (\k\Gamma_m)=1$, see \cite[Th. 16]{Au}. The statement about fractionally CY property is Lemma~\ref{lemma_fCYtensor} and Proposition~\ref{prop_DynkinCY}. 
The statement about Serre dimension is  Proposition~\ref{prop_fCYsdim}.
\end{proof}

Our main interest here is to compute Rouquier dimension and diagonal dimension of $B_m^n$ and to compare it with Serre dimension.   Recall that by Proposition~\ref{prop_rdimddim} we have $\Rdim\le \Ddim$. Also, for any $m$ and $n$ the algebra $B_m^n$ is a derived semi-orthogonal subalgebra in $B_m^{n+1}$, hence by Propositions~\ref{prop_rdimsubcat} and~\ref{prop_ddimsubcat}
$$\Rdim (B_m^n)\le \Rdim (B_m^{n+1}),\qquad \Ddim (B_m^n)\le \Ddim (B_m^{n+1}).$$
\begin{prop}
\label{prop_maintable}
For any $k\ge 1$ we have the following 
\begin{gather*}
\begin{array}{|c||c|c|c|c|c|c|c|c|}
\hline
\text{Algebra} & B_2 & B_2^2 & B_2^3 & B_2^4 & B_2^5 & B_2^6& B_2^{2k} & B_2^{3k} \\ \hline
\Sdim & 1/3 & 2/3 & 1& 4/3 & 5/3 & 2 & 2k/3 & k \\ \hline
\Rdim & 0 & 0 & 1& 1 & & 2 & \le k-1 & \ge k \\ 
\hline
\Ddim & 1 & 1 &  &  & &  & \le k & \ge k \\ 
\hline
\end{array}\\
\begin{array}{|c||c|c|c|c|c|c|c|}
\hline
\text{Algebra} & B_3 & B_3^2 & B_3^3 &  B_3^4 & B_3^{3k} & B_3^{3k+1} & B_3^{2k}\\ \hline
\Sdim & 1/2 & 1& 3/2& 2 & \frac{3k}2 &\frac{3k+1}2 & k\\ \hline
\Rdim & 0& 1& 1& 2& \le 2k-1 & \le 2k & \ge k\\ 
\hline
\Ddim & 1& & & & \le 2k & \le 2k+1 & \ge k\\ 
\hline
\end{array} 
\end{gather*}
\end{prop}
\begin{proof}
For Serre dimension see Lemma~\ref{lemma_sdimB}.

Let $i_1,\ldots,i_n\in \{0,1,\ldots, m-1\}$, then we denote 
$$P_{i_1\ldots i_n}:=P_{i_1}\otimes_\k\ldots\otimes_\k P_{i_n},$$
it is the indecomposable projective $B_m^n$-module, corresponding to the vertex $(i_1,\ldots,i_n)$ of quiver $\Gamma_m\times\ldots\times \Gamma_m$.

First, we explain the lower bounds. 

Projective modules 
$P_{100}, P_{010},P_{001}$, $P_{110},P_{101},P_{011}$ over $B_2^3$ form a strong  exception collection with  endomorphisms as follows: 
$$\xymatrix{P_{010}\ar[r] \ar[rd]& P_{110}  \\ 
P_{100} \ar[ru]\ar[rd] & P_{011} \\
 P_{001} \ar[r]\ar[ru]& P_{101}.}$$
Make mutations, see \eqref{eq_mutleft} and \eqref{eq_mutright}
\begin{align*}
E_1:=&L_{P_{010},P_{001}}(P_{011})[-1]=Cone(P_{010}\oplus P_{001}\to P_{011})[-1],\\
E_2:=&R_{P_{110},P_{101}}(P_{100})[1]=Cone(P_{100}\to P_{110}\oplus P_{101}).
\end{align*}
One checks that the resulting exceptional collection is also strong, its endomorphism algebra is the path algebra of the quiver 
$$\xymatrix{E_1\ar[r] \ar[rd] & P_{010}\ar[r] & P_{110}\ar[rd] & \\ 
& P_{001} \ar[r]& P_{101}\ar[r]& E_2}$$
without relations.
In particular,  one has $\Hom^{\bul}(E_1,E_2)=\k^2[0]$. 
It follows from Lemma~\ref{lemma_bondalec} now that the path algebra $K$ of the Kronecker quiver 
$$\xymatrix{\bul \ar@<0.3em>[r] \ar@<-0.3em>[r] & \bul}$$
is a derived semi-orthogonal subalgebra in $B_2^3$.
By Lemma~\ref{lemma_tensor}, the algebra $K^{\otimes k}$
is a derived semi-orthogonal subalgebra in 
$(B_2^3)^{\otimes k}=B_2^{3k}$. Note that the category $\Perf (K^{\otimes k})$  is 
equivalent to $D^b(\coh ((\P^1)^k))$.  We have 
$$\Rdim(B_2^{3k})\ge \Rdim (K^{\otimes k})=\Rdim(\coh((\P^1)^k))\ge k,$$
where the first inequality is by Proposition~\ref{prop_rdimsubcat} and the last inequality is by Proposition~\ref{prop_rdimgeom}.

Similarly, we prove that 
$\Rdim((B_3^2)^{\otimes k})\ge k$. By Lemma~\ref{lemma_tensor} and Corollary~\ref{cor_forest}, the algebra $B_3^2$ is derived Morita-equivalent to the algebra 
$$C=\k(\xymatrix{1\ar[r]& 0 & 2\ar[l]})\otimes \k(\xymatrix{1& 0\ar[r]\ar[l] & 2}).$$ 
The projective modules $P_{10}, P_{20},P_{01},P_{02}$ over $C$ have endomorphisms as follows: 
$$\xymatrix{P_{10} \ar[rd]\ar[r]& P_{01}\\ P_{20}\ar[r]\ar[ru] & P_{02}.}$$
Let $E:=R_{P_{01},P_{02}}(P_{20})[1]=Cone(P_{20}\to P_{01}\oplus P_{02})$, see \eqref{eq_mutright} Then the objects 
$P_{10},P_{01},P_{02},E\in \Perf C$ form a strong exceptional collection with the endomorphism algebra being the path algebra of the quiver 
$$\xymatrix{P_{10} \ar[rd]\ar[r]& P_{01}\ar[rd]&\\ &P_{02}\ar[r] & E}$$
without relations. As above, one checks that  $\Hom^{\bul}(P_{10},E)=\k^2[0]$. Hence, $K$ is a derived semi-orthogonal subalgebra of $C$ and thus of $B_3^2$.  Arguing as above, we see that $\Rdim B_3^{2k}=\Rdim (B_3^2)^{\otimes k}\ge k$. 

Inequalities $\Ddim(B_2), \Ddim(B_3)\ge 1$ follow from the fact that $B_2, B_3$ are not semi-simple, see \cite[Prop. 4.6]{EL}.

Now we establish the upper bounds. By Lemma~\ref{lemma_d4e6} below, 
the algebra $B_2^2$ is derived Morita-equivalent to the path algebra $\k D_4$ of the quiver  $D_4$: 
$$
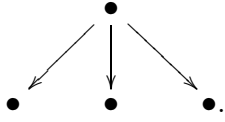
\xymatrix{ & \bul \ar[ld]\ar[d]\ar[rd] & \\ \bul & \bul & \bul.}$$
By Lemma~\ref{lemma_tensor}, the algebra $B_2^{2k}$ is derived Morita-equivalent to $(\k D_4)^{\otimes k}=(\k D_4)^{\otimes k-1}\otimes \k D_4$. 
Algebra $(\k D_4)^{\otimes k-1}$ is the path algebra  of the quiver $D_4^{k-1}=D_4\times\ldots\times D_4$ with commutativity relations. Note that the quiver $D_4^{k-1}$ has length $k-1$. Hence by Lemma~\ref{lemma_timesdyn} 
$$\Rdim(B_2^{2k})=\Rdim ((\k D_4)^{\otimes k-1}\otimes\k D_4)\le k-1.$$ 
In particular, $\Rdim(B_2^2)=0$, $\Rdim(B_2^4)\le 1$ and $\Rdim(B_2^6)\le 2$.  

For the diagonal dimension we have $\Ddim(B_2)=\Ddim(B_3)=\Ddim(\k D_4)=1$ by Proposition~\ref{prop_ddimnorels} and 
$$\Ddim(B_2^{2k})=\Ddim((\k D_4)^{\otimes k})\le k\cdot \Ddim(\k D_4)=k$$
by multiplicativity, see Proposition~\ref{prop_ddimmult}.



We note that  
$B_3$ is a derived semi-orthogonal subalgebra of $B_2^2$. Indeed, the  projective $B_2^2$-modules  $P_{00},P_{01},P_{11}$ form a strong exceptional collection with the endomorphism algebra being $B_3$, and Lemma~\ref{lemma_bondalec} can be applied. It follows from Lemma~\ref{lemma_tensor} that  for any $k\ge 1$ the algebra
 $B_3^{3k}$ is a derived semi-orthogonal subalgebra in  $B_3^{2k}\otimes B_2^{2k}$.
We bound above the Rouquier dimension of the latter algebra.
By Lemma~\ref{lemma_tensor} and Lemma~\ref{lemma_d4e6} below, 
the algebra 
$$B_3^{2k}\otimes B_2^{2k}=(B_3\otimes B_2)^{\otimes 2k}$$
is derived Morita-equivalent to the algebra 
$$(\k E_6)^{\otimes 2k}=(\k E_6)^{\otimes 2k-1}\otimes \k E_6,$$
where $E_6$ denotes the quiver 
$$
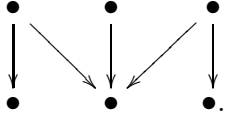
\xymatrix{\bul\ar[d]\ar[rd] & \bul \ar[d] & \bul\ar[d]\ar[ld]\\ \bul & \bul& \bul.}$$
Here the algebra $(\k E_6)^{\otimes 2k-1}$ is the path algebra of the quiver $E_6^{2k-1}=E_6\times\ldots\times E_6$ with commutativity relations. This quiver  has length $2k-1$. By   Proposition~\ref{prop_rdimsubcat} and Lemma~\ref{lemma_timesdyn} we get 
$$\Rdim(B_3^{3k})\le \Rdim (B_3^{2k}\otimes B_2^{2k})=\Rdim ((\k E_6)^{\otimes 2k-1}\otimes\k E_6)\le 2k-1.$$
In particular, $\Rdim(B_3^3)\le 1$.
Similarly  the algebra
$B_3^{3k+1}$ is a derived semi-orthogonal subalgebra in  $B_3^{2k}\otimes B_2^{2k}\otimes B_3$, which is derived Morita-equivalent to $(\k E_6)^{\otimes 2k}\otimes B_3$. Again, by  Proposition~\ref{prop_rdimsubcat} and Lemma~\ref{lemma_timesdyn}  we get  
$$\Rdim(B_3^{3k+1})\le \Rdim ((\k E_6)^{\otimes 2k}\otimes B_3)\le 2k.$$
In particular, $\Rdim(B_3^4)\le 2$.

For diagonal dimension, arguing as above we get
$$\Ddim(B_3^{3k})\le \Ddim((\k E_6)^{\otimes 2k})\le 2k\cdot \Ddim(\k E_6)=2k,$$
where the first inequality is by Proposition~\ref{prop_ddimsubcat}, the  second inequality is by Proposition~\ref{prop_ddimmult} and the last equality is by  Proposition~\ref{prop_ddimnorels}. Similarly,
$$\Ddim(B_3^{3k+1})\le \Ddim((\k E_6)^{\otimes 2k}\otimes B_3)\le 2k\cdot \Ddim(\k E_6)+\Ddim B_3=2k+1.$$
\end{proof}

\begin{lemma}
\label{lemma_d4e6}
Let $D_4$ and $E_6$ denote quivers of the corresponding Dynkin types with some orientation of arrows. Then the algebra $B_2^2$ is mutation-equivalent to the algebra $\k D_4$ and the algebra $B_3\otimes B_2$ is mutation-equivalent to the algebra $\k E_6$. Consequently, the above algebras are derived Morita-equivalent.
\end{lemma}
\begin{proof}
First note that by Corollary~\ref{cor_forest} it suffices to prove the statement for one  (convenient) orientation of arrows.
By the definition, $B_2^2= \k \Gamma_2\otimes \k\Gamma_2$ where $\Gamma_2$ is the quiver $0\to 1$. The category $\Perf B_2^2$ has a full strong exceptional collection of projective modules $P_{00},P_{01},P_{10},P_{11}$. Let $E=L_{P_{01},P_{10}}(P_{11})[-1]=Cone(P_{01}\oplus P_{10}\to P_{11})[-1]$, then the collection $P_{00},E,P_{01},P_{10}$ is also full and strong exceptional, its endomorphism algebra is the path algebra of the quiver
$$\xymatrix{P_{00}\ar[r] & E\ar[r] \ar[d] & P_{01}.\\ & P_{10} &}$$
Hence, $B_2^2$ is mutation-equivalent to $\k D_4$ by Definition~\ref{def_mutequiv}.
 
Now we prove the second statement. The algebra $B_3\otimes B_2$ is the path algebra of the quiver
$$\xymatrix{01\ar[r] & 11\ar[r] & 21\\
00\ar[r] \ar[u]& 10\ar[r] \ar[u]& \ 20\ar[u]\\}$$
with commutativity relations. Make mutations in the full strong exceptional collection of projective modules $P_{00},P_{01},P_{10},P_{11},P_{20},P_{21}$: let $E=R_{P_{01},P_{10}}(P_{00})[1]=Cone (P_{00}\to P_{01}\oplus P_{10})$ and
$F=L_{P_{11},P_{20}}(P_{21})[-1]=Cone(P_{11}\oplus P_{20}\to P_{21})[-1]$.
Then the exceptional collection
$P_{01},P_{10},E,F,P_{11},P_{20}$ is full and strong, it has the endomorphism algebra as follows:
$$\xymatrix{P_{01} \ar[r] & E\ar[r] & P_{11} &\\
& P_{10}\ar[r]\ar[u]& F\ar[r]\ar[u] & P_{20},}$$
where the square commutes.
Let $G=R_{E,F}(P_{10})[1]$, then 
the exceptional collection
$P_{01},E,F,G,P_{11},P_{20}$ is also full and strong, its endomorphism algebra is the path algebra of the quiver 
$$\xymatrix{P_{01} \ar[r] & E\ar[r] & G\ar[r] & P_{11} \\
& & F\ar[r]\ar[u] & P_{20}}$$
of Dynkin type $E_6$. Hence, $B_3\otimes B_2$ is mutation-equivalent to $\k E_6$ by definition.
\end{proof}

\begin{remark}
Note that in all examples in Proposition~\ref{prop_maintable} one has
$$\Rdim(A)=[\Sdim(A)].$$
It would be interesting to find out whether this is true in general.
\end{remark}

\section{Graded quivers with two vertices}
\label{section_twovertices}

Consider the 	quiver with two vertices $1,2$ and $n\ge 2$ arrows going from $1$ to $2$. Denote by $V$ the vector space spanned by arrows. Let 
$$A_V=\k e_1\oplus \k e_2\oplus V$$ 
be the corresponding path algebra. In this section we consider $A_V$ as a graded algebra by introducing some weights on arrows (and thus some $\Z$-grading on $V$). Moreover, we consider $A_V$ as a dg algebra with zero differential. Such dg algebra is smooth and compact. The category $\Perf A_V$ has full exceptional collection $P_1=e_1A_V,P_2=e_2A_V$
of right graded $A_V$-modules. Therefore  $\Ddim(\Perf A_V)\le  1$ by \cite[Prop. 4.13]{EL}. On the other hand, we have $\Rdim(\Perf A_V)>0$. Indeed, otherwise by Lemma~\ref{lemma_rdimpositive} we get $\LSdim(\Perf A_V)=\USdim(\Perf A_V)$ what contradicts to Proposition~\ref{prop_ww} below.  It follows that 
$$\Rdim(\Perf A_V)=\Ddim(\Perf A_V)=1.$$

In order to find lower and upper Serre dimensions of $\Perf A_V$, we compute explicitly derived tensor powers of the Serre bimodule $A_V^*$. To formulate the answer, we introduce some notation.

\begin{definition}
\label{def_psi}
Let $V$ be a vector space over $\k$. Denote  $AT_0(V):=\k$ and for $n\ge 1$
$$AT_n(V):=\underbrace{V\otimes V^*\otimes V\otimes V^*\otimes\ldots}_{n\,\,\text{factors}}$$
There are $n-1$ trace maps
$$tr_1,\ldots,tr_{n-1}\colon AT_n(V)\to AT_{n-2}(V),$$
where $tr_i=\id^{\otimes(i-1)}\otimes tr \otimes \id^{\otimes(n-i-1)}$
and $tr$ denotes the pairing $V\otimes V^*\to \k$ or $V^*\otimes V\to\k$.
We put $\psi_{-1}(V):=0$, $\psi_0(V):=\k$, $\psi_1(V):=V$ and for $n\ge 2$
$$\psi_n(V):=\cap_{i=1}^{n-1}\ker tr_i\subset AT_n(V).$$
If $V$ is a graded vector space then $V^*,AT_n(V)$ and $\psi_n(V)$  also carry a natural grading.
\end{definition}

We observe that a right graded $A_V$-module $M$ is given by the following data: two graded vector spaces
$M_1,M_2$ and a homogeneous homomorphism of graded vector spaces $M_1\gets M_2\otimes V$. We will denote such graded module by $(M_1\Leftarrow M_2)$. In particular, the indecomposable projective and injective right graded $A_V$-modules are
$$P_1=(\k\Leftarrow 0),\quad P_2=(V\Leftarrow \k),\qquad I_1=(\k\Leftarrow V^*),\quad I_2=(0\Leftarrow \k).$$

\begin{lemma}
\label{lemma_psi}
Suppose $\dim V\ge 2$. Let $A:=A_V$, then

\begin{enumerate}
\item[($A_i$)] There are the following exact sequences of graded vector spaces for any $i\ge 0$:

\begin{align*}
& i=2k-1: & &0\to \psi_{2k}(V)\xra{\alpha^V_{2k-1}} \psi_{2k-1}(V)\otimes V^*\xra{\beta^V_{2k-1}} \psi_{2k-2}(V)\to 0,\\
& i=2k: & 
 &0\to \psi_{2k+1}(V)\xra{\alpha^V_{2k}} \psi_{2k}(V)\otimes V\xra{\beta^V_{2k}} \psi_{2k-1}(V)\to 0. \\
\end{align*}

\item[($B_i$)]
There are the following exact sequences of graded $A$-modules for any $i\ge 0$:
\begin{align*}
& i=2k-1: & &0\to (\psi_{2k}(V^*)\Leftarrow 0) \to (\psi_{2k-1}(V^*)\otimes V\Leftarrow \psi_{2k-1}(V^*))\to (\psi_{2k-2}(V^*)\Leftarrow \psi_{2k-1}(V^*))\to 0,\\
& i=2k: & 
&0\to (\psi_{2k+1}(V)\Leftarrow 0) \to (\psi_{2k}(V)\otimes V\Leftarrow \psi_{2k}(V))\to (\psi_{2k-1}(V)\Leftarrow \psi_{2k}(V))\to 0.
\end{align*}

\item[($C_i$)]
For any $i\ge 0$ there exist the following isomorphisms in $\Perf A$:
\begin{align*}
& i=2k-1: & 
P_1\otimes^L_A(A^*)^{\otimes^L_Ak}&\cong (\psi_{2k-2}(V^*)\Leftarrow \psi_{2k-1}(V^*))[k-1];\\
& i=2k: &
P_2\otimes^L_A(A^*)^{\otimes^L_A(k+1)}&\cong (\psi_{2k-1}(V)\Leftarrow \psi_{2k}(V))[k],
\end{align*}
where the structure maps in the right-hand sides are $\beta_{2k-1}^{V^*}$ and $\beta_{2k}^V$.
\end{enumerate}
\end{lemma}

\begin{proof}
The maps $\alpha^V_i$ and $\beta^V_i$  are defined as follows. For $i=2k$, consider the composite map
$$\psi_{2k}(V)\otimes V\to AT_{2k+1}(V)\xra{tr_{2k}} AT_{2k-1}(V),$$
its image lies in $\psi_{2k-1}(V)$. By the definition, the kernel of the above composite map is $\psi_{2k+1}(V)$. Hence sequences in $(A_i)$ are defined and are left exact by the definition, similarly for  odd $i$. The statement is that $\beta^V_i$ is surjective.

Homomorphisms of modules in $(B_{2k})$ are given by the commutative diagram
$$\xymatrix{
\psi_{2k+1}(V)\ar[d]^{\alpha_{2k}^{V}} && 0\otimes V\ar[ll]\ar[d] \\
\psi_{2k}(V)\otimes V\ar[d]^{\beta_{2k}^{V}} && \psi_{2k}(V)\otimes V\ar@{=}[d] \ar@{=}[ll] \\
\psi_{2k-1}(V) && \psi_{2k}(V)\otimes V, \ar[ll]_{\beta_{2k}^{V}}
}$$
and similarly for $(B_{2k-1})$. Therefore $(A_i)$ is equivalent to $(B_i)$ for any $i$.

We will prove $A_i$, $B_i$ and $C_i$ simultaneously by induction in $i$. First, we note that $(A_0)$ and $(A_1)$ are clearly true. 
As $P_i$ is a projective $A$-module, 
$P_i\otimes^L_AA^*=P_i\otimes_AA^*=e_iA^*\cong I_i$ by \eqref{eq_SPI}. Therefore
$P_1\otimes^L_AA^*\cong (\k\Leftarrow V^*)$ and 
$P_2\otimes^L_AA^*\cong (0\Leftarrow \k)$, that is, $(C_0)$ and $(C_1)$ hold.  


Now, for any $i\ge 1$, we prove that $(B_i)+(C_i)\Rightarrow (A_{i+1})+(C_{i+2})$. It will  follow that all $(A_i)$, $(B_i)$ and $(C_i)$ are true. 

Assume $i$ is even, $i=2k$. The case of odd $i$ can be done similarly.
By $(C_{2k})$, we have an isomorphism
$$P_2\otimes^L_A(A^*)^{\otimes_A^L(k+1)}\cong (\psi_{2k-1}(V)\Leftarrow \psi_{2k}(V))[k].$$
By $(B_{2k})$, the module $(\psi_{2k-1}(V)\Leftarrow \psi_{2k}(V))$ is quasi-isomorphic to the cone of the morphism
$$(\psi_{2k+1}(V)\Leftarrow 0) \to (\psi_{2k}(V)\otimes V\Leftarrow \psi_{2k}(V)).$$
This morphism is nothing but
$$\psi_{2k+1}(V)\otimes P_1\to \psi_{2k}(V)\otimes P_2,$$
both terms are projective $A$-modules.
Tensoring by $A^*$, we get 
\begin{multline*}
(\psi_{2k-1}(V)\Leftarrow \psi_{2k}(V))\otimes^L_AA^*\cong
Cone(\psi_{2k+1}(V)\otimes P_1\to \psi_{2k}(V)\otimes P_2)\otimes_AA^*\cong \\
\cong Cone(\psi_{2k+1}(V)\otimes I_1\to \psi_{2k}(V)\otimes I_2)\cong
Cone(\psi_{2k+1}(V)\otimes (\k\Leftarrow V^*)\xra{f} \psi_{2k}(V)\otimes (0\Leftarrow \k)).
\end{multline*}
Since $A$ has global dimension one, one has a quasi-isomorphism
$$Cone (f)\cong (\ker f)[1]\oplus \coker f.$$ 
The map $f$ is given by $(0\Leftarrow \beta_{2k+1}^V)$, hence
$$\ker f\cong (\psi_{2k+1}(V)\Leftarrow \psi_{2k+2}(V))$$
and 
$$ \coker f=(0\Leftarrow \coker\beta_{2k+1}^V).$$
Recall that the functor $-\otimes^L_AA^*$ on $\Perf A$ is a Serre functor, in particular, it is an equivalence. It follows that $Cone (f)$ is indecomposable in $\Perf A$, and since $\ker f\ne 0$ (here we use that $\dim (V)\ge 2$), we get $\coker f=0$. It follows now that $\beta_{2k+1}^V$ is surjective, $(A_{2k+1})$ is proven. Finally, 
\begin{multline*}
P_2\otimes^L_A(A^*)^{\otimes_A^L(k+2)} \cong (\psi_{2k-1}(V)\Leftarrow \psi_{2k}(V))\otimes^L_AA^*[k]\cong\\
\cong Cone(f)[k]\cong (\ker f)[k+1]\cong (\psi_{2k+1}(V)\Leftarrow \psi_{2k+2}(V))[k+1],
\end{multline*}
and $(C_{2k+2})$ is proved.
\end{proof}

\begin{lemma}
\label{lemma_kw}
Suppose $V$ is a graded vector space and $\dim V\ge 2$. Then for any $k\ge 1$ one has (see Definition~\ref{def_w})
$$\sup \psi_{2k-1}(V)=kw(V)+\inf V,\quad \inf \psi_{2k-1}(V)=-kw(V)+\sup V$$
and 
$$\sup \psi_{2k}(V)=kw(V),\quad \inf \psi_{2k}(V)=-kw(V).$$
\end{lemma}
\begin{proof}
Denote $s:=\sup V, i:=\inf V$. Then 
\begin{align*}
\sup AT_{2k-1}(V)&=k\sup V+(k-1)\sup V^*=ks-(k-1)i=kw(V)+i,\\
\inf AT_{2k-1}(V)&=k\inf V+(k-1)\inf V^*=ki-(k-1)s=-kw(V)+s,\\
\sup AT_{2k}(V)&=k\sup V+k\sup V^*=ks-ki=kw(V),\\
\inf AT_{2k}(V)&=k\inf V+k\inf V^*=ki-ks=-kw(V).
\end{align*}
To prove that we have the same bounds for $\psi_i(V)\subset AT_i(V)$, we demonstrate that $\psi_i(V)$ contains elements of the maximal  and the minimal possible degree in $AT_i(V)$. Choose a basis $e_1,\ldots,e_n$ in $V$ compatible with the grading such that 
$\deg e_1=s, \deg e_n=i$. Let $e^1,\ldots,e^n\in V^*$ denote the dual basis, we have $\deg e^1=-s, \deg e^n=-i$. Let 
$$x=e_1\otimes e^n \otimes e_1\otimes e^n\otimes\ldots\otimes e^n, 
y=e_n\otimes e^1 \otimes e_n\otimes e^1\otimes\ldots\otimes e^1,$$ 
then $x,y\in\psi_{2k}(V)$ and
$\deg x=ks-ki=kw(V)$, $\deg y=ki-ks=-kw(V)$.
Similarly for $\psi_{2k-1}$.
\end{proof}

\begin{prop}
\label{prop_ww}
Let $V$ be a graded vector space with $\dim V\ge 2$ and $w(V)=w$ (see Definition~\ref{def_w}). Then for the corresponding graded algebra $A_V$ (with zero differential) one has
$$\LSdim (\Perf A_V)=1-w,\qquad \USdim(\Perf A_V)=1+w.$$
\end{prop}
\begin{proof}
By  \cite[Prop. 5.5]{EL}, we have 
\begin{equation}
\label{eq_limdim}
\LSdim \Perf A_V=\lim_m \frac{-\sup (A_V^*)^{\otimes^L_{A_V}m}}m, \quad
\USdim \Perf A_V=\lim_m \frac{-\inf (A_V^*)^{\otimes^L_{A_V}m}}m.
\end{equation}
By Lemma~\ref{lemma_psi}, we have isomorphisms in $\Perf A_V$: 
\begin{multline*}
(A_V^*)^{\otimes^L_{A_V}m}\cong \left(P_1\otimes^L_{A_V}(A_V^*)^{\otimes^L_{A_V}m}\right)\oplus \left(P_2\otimes^L_{A_V}(A_V^*)^{\otimes^L_{A_V}m}\right)\cong \\
\cong ((\psi_{2m-2}(V^*)\Leftarrow \psi_{2m-1}(V^*))\oplus (\psi_{2m-3}(V)\Leftarrow \psi_{2m-2}(V)))[m-1].
\end{multline*}
It follows from Lemma~\ref{lemma_kw} that 
\begin{multline*}
\sup ((A_V^*)^{\otimes^L_{A_V}m})= \max((m-1)w,(m-1)w-\inf V,(m-1)w+\inf V)+1-m=\\
=(m-1)w+|\inf V|+1-m
\end{multline*}
and
\begin{multline*}
\inf ((A_V^*)^{\otimes^L_{A_V}m})=\min(-(m-1)w,-(m-1)w-\sup V,-(m-1)w+\sup V)+1-m=\\
=-(m-1)w-|\sup V|+1-m.
\end{multline*}
Now using \eqref{eq_limdim} we get the statement.
\end{proof}

\begin{remark}
Let $\TT=\langle E_1,E_2\rangle$ be a dg enhanced triangulated category, generated by an exceptional pair. Then $\TT\cong \Perf A_V$ where $V=\Hom^\bul(E_1,E_2)$. Indeed, by general theory, $\TT$ is equivalent to $\Perf \AA$ where $\AA=R\End(E_1\oplus E_2)$ is the dg endomorphism algebra. Clearly, $\AA$ is formal: there exists a quasi-isomorphism
$A_V\cong H^{\bul}(\AA)\to \AA$. Hence $\Perf \AA\cong\Perf A_V$.
\end{remark}

\section{Some other examples}
\label{section_other}

In this last section we consider some (to be precise, three) examples of path algebras with relations and  compute their Rouquier, diagonal and Serre dimension. 

We denote by $[\ldots P\to \underline{Q}\to R\to S\ldots ]$ the corresponding complex with $Q$ placed in degree $0$.

\begin{example}
\label{example_1}
Consider the quiver 
$$\xymatrix{0 \bul  \ar[rr]^x \ar@/_1pc/[rrrr]_z && \bul 1 \ar[rr]^y && \bul 2
}.$$
Denote the path algebra with relation $yx=0$ by $A$.

Clearly $\gldim(A)=2$.

Projective and injective modules over $A$ are
$$\xymatrix{
P_0\colon & (\k && 0 && 0)\\
P_1\colon & (\k  && \k \ar[ll] && 0)\\
P_2\colon & (\k && \k && \k \ar[ll] \ar@/^1pc/[llll])
}$$
and
$$\xymatrix{
I_0\colon & (\k && \k \ar[ll]&& \k  \ar@/^1pc/[llll])\\
I_1\colon & (0 && \k && \k  \ar[ll])\\
I_2\colon & (0 && 0 && \k)
}$$
(where arrows denote identity maps and missing arrows act by zero).

We have $\Ddim (A)\le 1$ by Proposition~\ref{prop_ddimradical} since 
$R(A)^2=0$. On the other hand, assume $\Rdim(A)=0$. Then  $\LSdim(A)=\USdim(A)$ by Lemma~\ref{lemma_rdimpositive} what is false, see below. 
Therefore
$$\Rdim (A)=\Ddim(A)=1.$$

Let $M$ be the module 
$$\xymatrix{(\k && 0 && \k  \ar@/^1pc/[llll])}.$$
Its projective resolution and injective resolutions are  
$$[P_0\to P_1\to \underline{P_2}] \quad\text{and}\quad [\underline {I_0}\to I_1\to I_2].$$
Hence by \eqref{eq_SPI} 
\begin{equation}
\label{eq_MM}
S(M)\cong [I_0\to I_1\to \underline{I_2}]\cong M[2].
\end{equation}
It follows that $\USdim(A)\ge 2$. Since $\USdim (A)\le \gldim (A)=2$ (Proposition~\ref{prop_serregldim}), we get
$$\USdim (A)=2.$$

Now we prove that 
$$\LSdim(A)=1/2.$$
For this we perform some direct calculation which comprise two following lemmas.
\begin{lemma}
\label{lemma_xyz1} 
We have for any $k\ge 1$ and $i\le 2k-1$
\begin{enumerate}
\item $H_i(S^{k+1}(P_0))\cong H_i(S^k(S_1))$; 
\item $H_i(S^{k+1}(S_1))\cong H_{i-1}(S^k(P_0))$;
\item $H_i(S^{k+1}(P_2))\cong H_{i-1}(S^k(P_0))$.
\end{enumerate}
\end{lemma}
\begin{proof}
1) One has an exact sequence of $A$-modules
$$0\to M\to I_0\to S_1\to 0$$
and thus a distinguished triangle $S^k(M)\to S^k(I_0)\to S^k(S_1)\to S^k(M)[1]$ for any $k\ge 1$. Since $S(P_0)\cong I_0$ and \eqref{eq_MM} this triangle is isomorphic to  
$$M[2k]\to S^{k+1}(P_0)\to S^k(S_1)\to M[2k+1].$$
The associated exact sequence in homology yields the statement.

2) The module $S_1$ has projective resolution $[P_0\to \underline{P_1}]$, hence
$S(S_1)\cong [I_0\to \underline{I_1}]$. The latter complex is quasi-isomorphic to the total complex $\Tot B$ of the bicomplex
$$B={\xymatrix{P_0\ar[r] & P_1\ar[r] & P_2 &&\\
&& P_0\ar[u] \ar[r]& P_1\ar[r] & \underline{P_2} \\ &&&& P_0,\ar[u]}}
$$
where the underlined term has degree $(0,0)$. Let $\widetilde B\subset B$ be the subbicomplex of terms $B^{ij}$ with $j\ge 0$. Let $\widetilde M:=\Tot \widetilde B$. Recall that $M\cong [P_0\to P_1\to \underline{P_2}]$. One has a distinguished triangle  
$$M[1]\to \widetilde M\to M\to M[2].$$ 
Since $\Ext^2(M,M)=\k$, it follows that $S(\widetilde M)\cong \widetilde M[2]$.
Also note that one has distinguished triangles
$\widetilde M\to \Tot B\to P_0[1]\to \widetilde M[1]$ and thus
$S^k(\widetilde M)\to S^k(\Tot B)\to S^k(P_0)[1]\to S^k(\widetilde M)[1]$. By the above, the latter triangle is isomorphic to
$$\widetilde M[2k]\to S^{k+1}(S_1)\to S^k(P_0)[1]\to \widetilde M[2k+1].$$
The statement now follows from looking at the associated exact sequence in homology (note that $H_i(\widetilde M)=0$ for $i<0$).

3) The proof is similar to 1). We use the exact sequence of $A$-modules
$$0\to P_0\to M\to I_2\to 0$$
and the isomorphism $S(P_2)\cong I_2$.
\end{proof}

\begin{lemma}
\label{lemma_xyz2} 
We have 
\begin{align}
\label{ppp1}
&H_{k-1}(S^{2k}(P_0))=I_2, & \quad   & H_{i}(S^{2k}(P_0))=0 \quad && \text{for any}\quad k\ge 1, i<k-1;\\
\label{ppp2}
&H_{k}(S^{2k+1}(P_0))=I_0, & \quad   & H_{i}(S^{2k+1}(P_0))=0\quad && \text{for any} \quad k\ge 0, i<k;\\
\label{ppp3}
&S^{2k}(P_1)=P_1[k], &\quad & S^{2k+1}(P_1)=I_1[k] \quad && \text{for any}\quad k\ge 0; \\
\label{ppp4}
&H_{k}(S^{2k}(P_2))=I_0, & \quad  & H_{i}(S^{2k}(P_2))=0\quad && \text{for any}\quad k\ge 1, i<k;\\
\label{ppp5}
&H_{k}(S^{2k+1}(P_2))=I_2, & \quad  & H_{i}(S^{2k+1}(P_2))=0\quad && \text{for any}\quad k\ge 0, i<k.
\end{align}
\end{lemma}
\begin{proof}
We prove \eqref{ppp1} by induction in $k$. Let $k=1$ then $H_i(S^2(P_0))=H_i(S(S_1))$ for any $i\le 1$ by Lemma~\ref{lemma_xyz1}. We have $S(S_1)=S([P_0\to \underline {P_1}])=[I_0\to \underline {I_1}]$. In particular, $H_i(S(S_1))=0$ for $i<0$ and $H_0(S(S_1))=\coker (I_0\to I_1)\cong I_2$. Now suppose $k\ge 2$ and $i\le k-1$. We have
\begin{align*}
H_i(S^{2k}(P_0)) & = H_i(S^{(2k-1)+1}(P_0))  && \\
&= H_i(S^{2k-1}(S_1)) &&\text{by Lemma~\ref{lemma_xyz1}(1) since} \quad i\le 2(2k-1)-1\\
&= H_i(S^{(2k-2)+1}(S_1)) && \\
&= H_{i-1}(S^{2(k-1)}(P_0)) && \text{by Lemma~\ref{lemma_xyz1}(2) since}\quad i\le 2(2k-2)-1\\
&=\begin{cases} I_2& \text{for}\quad i=k-1,\\ 0& \text{for} \quad i<k-1.\end{cases} &&
\end{align*}

The proof of \eqref{ppp2} is similar.

For the proof of \eqref{ppp3} it suffices to note that $I_1\cong [P_0\to\underline{P_2}]$, $S(I_1)\cong [I_0\to\underline{I_2}]$ and $P_0\cong [\underline {I_0}\to I_2]$.

Statements \eqref{ppp4} and \eqref{ppp5} follow from \eqref{ppp2}, \eqref{ppp1} and Lemma~\ref{lemma_xyz1}(3).
\end{proof}

Lemma~\ref{lemma_xyz2} implies that for any $k\ge 0$ one has isomorphisms of right $A$-modules
$$H^{-k}((A^*)^{\otimes^L_A(2k+1)})\cong A^*,\quad H^{i}((A^*)^{\otimes^L_A(2k+1)})=0\quad\text{for}\quad i>-k.$$
It follows that $\sup ((A^*)^{\otimes^L_A(2k+1)})=-k$ and thus $\LSdim(A)=1/2$ by \eqref{eq_infsup}.
\end{example}

\begin{example}
\label{example_2}

Now consider the non-ordered quiver 
$$
\xymatrix{0 \bul  \ar[rr]^x  && \bul 1 \ar[rr]^y && \bul 2 \ar@/^1pc/[llll]^z}
$$
with relations $zy=xz=0$. Denote the quotient algebra by $A$.
One has $\gldim(A)=3$.

Projective and injective modules over $A$ are
$$\xymatrix{
P_0\colon & (\k \ar@/_1pc/[rrrr]&& 0 && \k)\\
P_1\colon & (\k  && \k \ar[ll] && 0)\\
P_2\colon & (\k && \k \ar[ll]&& \k \ar[ll])
}$$
and
$$\xymatrix{
I_0\colon & (\k && \k \ar[ll]&& \k  \ar[ll])\\
I_1\colon & (0 && \k && \k  \ar[ll])\\
I_2\colon & (\k \ar@/_1pc/[rrrr]&& 0 && \k).
}$$

We have  
$$S(P_0)\cong I_0\cong P_2,\quad S(P_2)\cong I_2\cong P_0.$$

Projective resolution of $S(P_1)=I_1$ is 
$$[P_1\to P_2\to P_0\to \underline{P_2}].$$
Hence, 
$$S^2(P_1)\cong S(I_1)\cong [I_1\to P_0\to P_2\to \underline{P_0}].$$ 
Further, 
$$S^3(P_1)\cong [I_1\to P_0\to P_2\to P_0\to P_2\to P_0\to \underline{P_2}].$$
Iterating, we see that for any $k\ge 1$
$$S^{2k+1}(P_1)\cong 
[I_1\to \underbrace{P_0\to P_2\to P_0\to P_2\to \ldots \to P_0\to \underline{P_2}}_{\text{$P_0$ occurs $3k$ times}}].$$
In particular, $H_{6k}(S^{2k+1}(P_1)), H_0(S^{2k+1}(P_1))\ne 0$. 
It follows that 
$$H_{6k}((A^*)^{\otimes^L_A 2k+1}), H_0((A^*)^{\otimes^L_A 2k+1})\ne 0$$ 
and thus 
$$\inf ((A^*)^{\otimes^L_A 2k+1})\le -6k, \sup ((A^*)^{\otimes^L_A 2k+1})=0.$$
It follows from \eqref{eq_infsup} that $\USdim(A)\ge 3$ and $\LSdim(A)=0$. By Proposition~\ref{prop_serregldim} we have $\USdim(A)\le \gldim A= 3$, and we get 
$$\LSdim(A)=0, \USdim(A)=3.$$

Note that we have a full exceptional collection in $D^b(\modd A)$, which is not strong:
$$D^b(\modd A)=(S_2,P_0,P_1).$$
The category $\langle P_0,P_1\rangle$ is equivalent to the category $D^b(\modd \k A_2)$ because $\Hom^{\bul}(P_0,P_1)=\k[0]$, where $A_2$ denotes a certain Dynkin quiver. The latter category has Rouquier dimension $0$ by Proposition~\ref{prop_dynkin}, as well as the category $\langle S_2\rangle$. 
It follows that $\Rdim(A)\le 1$. Since $\LSdim A\ne \USdim A$ we deduce from Lemma~\ref{lemma_rdimpositive} that
$$\Rdim (A)=1.$$ 
By Proposition~\ref{prop_ddimexcoll} we have $\Ddim(A)\le 2$. We do not know whether 
$\Ddim(A)=1$ or~$2$.

\end{example}

\begin{example}
\label{example_3}
Let $A$ be the Auslander algebra of the algebra $\k\langle t\rangle$ of dual numbers. Algebra $A$ is isomorphic to the path algebra of the quiver
$$
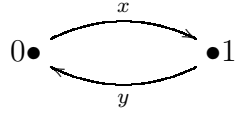
\xymatrix{0 \bul  \ar@/^1pc/[rr]^x  && \bul 1 \ar@/^1pc/[ll]^y }
$$
with relation $xy=0$. We have $\gldim(A)=2$.
Calculations show that $S(P_0)\cong I_0\cong P_0$ and $S(P_1)\cong I_1\cong [P_1\to P_0\to \underline{P_0}]$.
Consequently, 
$$S^n(P_1)\cong [P_1\to\underbrace{P_0\to P_0\to\ldots\to P_0\to \underline{P_0}}_{2n}]$$
and $H_i(S^n(P_1))\ne 0$ exactly for  $0\le i\le 2n$.
Hence 
$$\inf ((A^*)^{\otimes^L_A n})=-2n, \sup ((A^*)^{\otimes^L_A n})=0.$$
It follows (see \eqref{eq_infsup}) that 
$$\LSdim(A)=0,\quad \USdim(A)=2.$$

Modules $S_0, P_1$ form a full exceptional collection in $\Perf A$ (which is not strong). It follows from Proposition~\ref{prop_ddimexcoll} that $\Rdim(A)\le \Ddim(A)\le 1$.
Since $\LSdim(A)\ne \USdim (A)$ we have $\Rdim(A)\ne 0$ by Lemma~\ref{lemma_rdimpositive}.
Consequently, 
$$\Rdim(A)=\Ddim(A)=1.$$
\end{example}

\end{document}